\documentclass{amsart}
\usepackage{amsmath}
\usepackage{amssymb}
\usepackage{latexsym}
\usepackage{verbatim}
\usepackage{graphics}
\usepackage{amsfonts}
\usepackage{graphicx}
\usepackage{epstopdf}
\theoremstyle{plain}

\newtheorem{prop}{Proposition}
\newtheorem{lem}{Lemma}

\newtheorem{claim}{Claim}

\def\lk{{\rm lk}}

\theoremstyle{remark}

\newtheorem*{question}{Question}
\newtheorem*{rem}{Remark}

\begin{document}

\title{Counting links and knots in complete graphs}
\date{\today}

\author{Loren Abrams}
\email{loren.abrams@gmail.com}

\author{Blake Mellor}
\address{Mathematics Department\\
   	Loyola Marymount University\\
   	Los Angeles, CA  90045-2659}
\email{blake.mellor@lmu.edu}

\author{Lowell Trott}
\address{Department of Computer Science\\
	Bren School of Information and Computer Sciences\\
	University of California, Irvine 92697}
\email{ltrott@uci.edu}

\keywords{intrinsically linked graphs, intrinsically knotted graphs, spatial graphs}
\subjclass[2000]{05C10; 57M25}

\thanks{This research was supported in part by NSF grant DMS - 0905687.}

\begin{abstract}

We investigate the minimal number of links and knots in embeddings of complete partite graphs in $S^3$.  We provide exact values or bounds on the minimal number of links for all complete partite graphs with all but 4 vertices in one partition, or with 9 vertices in total.  In particular, we find that the minimal number of links in an embedding of $K_{4,4,1}$ is 74.  We also provide exact values or bounds on the minimal number of knots for all complete partite graphs with 8 vertices.

\end{abstract}

\maketitle



\section{Introduction} \label{S:intro}

The study of links and knots in spatial graphs began with Conway and Gordon's seminal result that every embedding of $K_6$ in $S^3$ contains a non-trivial link and every embedding of $K_7$ in $S^3$ contains a non-trivial knot \cite{cg}.  Their result sparked considerable interest in {\it intrinsically linked} and {\it intrinsically knotted} graphs -- graphs with the property that every embedding in $\mathbb{R}^3$ contains a pair of linked cycles (respectively, a knotted cycle).  Robertson, Seymour and Thomas \cite{rst} gave a Kuratowski-type classification of intrinsically linked graphs, showing that every such graph contains one of the graphs in the {\it Petersen family} as a minor (see Figure~\ref{F:petersen}).  There is, as yet, no such classification for intrinsically knotted graphs; and since there are dozens of known minor-minimal intrinsically knotted graphs (see \cite{fo2, fo3, ks}), any such classification will be far more complex.

    \begin{figure} [htpb]
    $$\includegraphics{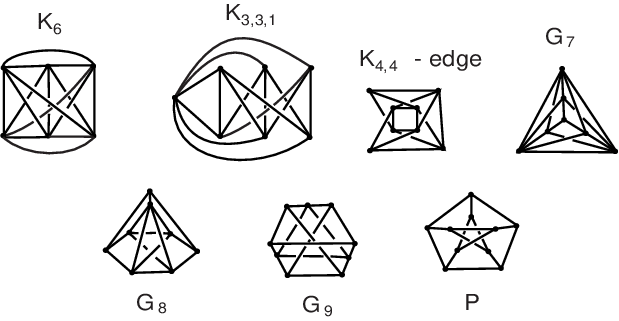}$$
    \caption{The Petersen family of graphs} \label{F:petersen}
    \end{figure} 

However, while Robertson, Seymour and Thomas answered the question of {\it which} graphs are intrinsically linked, they did not address {\it how} they are linked, and how complicated the linking must be.  In this paper, we measure the ``complexity" of a graph with respect to intrinsic linking (respectively, intrinsic knotting) by the minimal number of links (respectively, knots) in any embedding of the graph (denoted $mnl(G)$ or $mnk(G)$).

This is not the only possible measure of complexity.  Rather than counting the number of links or knots, one could focus instead on the complexity of the individual links or knots.  Flapan \cite{fl} has given examples of graphs which must contain links with large linking numbers and knots with large $a_2$ (the second coefficient of the Conway polynomial), and Flapan, Foisy, Naimi and Pommersheim \cite{ffnp} constructed graphs whose embeddings must contains links with many components.  Recently, the second author, with Flapan and Naimi \cite{fmn}, has generalized these results to show that there are graphs whose embedding must contain a link which is arbitrarily complex as measured by {\it both} the pairwise linking numbers {\it and} the size of the second coefficient of the Conway polynomial of the components.

The notion of $mnl(G)$ was introduced by Tom Fleming and the second author in \cite{fm}, where they investigated the minimum number of links in complete partite graphs on 7 or 8 vertices.   We extend this investigation to complete partite graphs on 9 vertices, and also for several general families of complete partite graphs.  We also investigate the minimum number of knots in complete partite graphs on 8 vertices.  The only previous results in this area are bounds given by Hirano \cite{hi} (improving on results of Blain et. al. \cite{bbhlf}) for the minimal number of knotted Hamiltonian cycles in $K_8$.

In general, finding the minimum number of links or knots in a graph requires determining both a lower bound and and upper bound, and then working to bring these bounds together.  Upper bounds are established by examining particular embeddings and counting the number of links (or knots) in the embedding.  While simple in theory, this is very difficult in practice -- even relatively simple graphs can have hundreds or thousands of cycles which need to be checked; and each time the embedding is changed in hopes of reducing the number of links or knots, the computation must be repeated.  Clearly, this task is best done by a computer, and much of our effort has been to develop a program {\em Gordian} \cite{amt} to do these computations.  The use inputs a file containing the crossing data for the embedding, and the program will then find all pairs of linked cycles with nontrivial linking number and all knotted cycles where the second coefficient of the Conway polynomial is non-zero.  For our purposes, these invariants were largely sufficient.

Lower bounds are generally determined by looking for subgraphs for which the minimum number of links (or knots) is known.  Of course, a given link (or knot) may appear in several different subgraphs, so the combinatorial analysis can become quite complex.  In this paper, the most ambitious example is the proof that the minimum number of links for $K_{4,4,1}$ is 74.

In Section \ref{S:prelim}, we provide definitions and notation, and recall some useful results from \cite{fm}.  In Section \ref{S:general}, we determine the minimum number of links in complete partite graphs with all but four vertices in one partition (in the case of $K_{n,1,1,1,1}$ we find upper and lower bounds); these results are summarized in Table \ref{Ta:general}.

\begin{table} [htpb]
	\begin{center}
	\begin{tabular}{|l|c|}
		\hline
		$G$ & $mnl(G)$ \\ \hline
		$K_{n,4} {\rm \ or\ } K_{n,2,2}$ & $2\dbinom{n}{4}$ \\ \hline
		$K_{n,3,1} {\rm \ or\ } K_{n,2,1,1}$ & $\dbinom{n}{3} + 2\dbinom{n}{4}$ \\ \hline
		$K_{n,1,1,1,1}$ & $\displaystyle{2\dbinom{n}{4} + 2\dbinom{n}{3} + \delta,\ {\rm where\ } \left\lceil\frac{n^2-n}{6}\right\rceil \leq \delta \leq  \left\lceil\frac{n^2-2n}{4}\right\rceil}$ \\ \hline
	\end{tabular}
	\end{center}
	\caption{Minimum number of links for some families of complete partite graphs} \label{Ta:general}
\end{table}

In Section \ref{S:9vertex} we find exact values or upper bounds for the minimum number of links in all intrinsically linked complete partite graphs on 9 vertices.  In particular, in Section \ref{SS:K441} we prove our most difficult result: that the minimum number of links for $K_{4,4,1}$ is 74.  The results are summarized in Table \ref{Ta:9vertex}.  We only list the intrinsically linked graphs.

\begin{table} [htpb]
	\begin{center}
	\begin{tabular}{| l | c | c | l | c |}
		\cline{1-2} \cline{4-5}
		$G$ & $mnl(G)$ & & $G$ & $mnl(G)$\\ \cline{1-2} \cline{4-5}                                                                                                                                                                                                                                                                                                                                                                                                                               
		$K_{5,4}$ & 10 & & $K_{3,3,3}$ & $\leq 248$ \\ \cline{1-2} \cline{4-5}
		$K_{5,3,1}$ & 20 & & $K_{3,3,2,1}$ & $\leq 386$\\ \cline{1-2} \cline{4-5}
		$K_{5,2,2}$ & 10 & & $K_{3,3,1,1,1}$ & $\leq 555$\\ \cline{1-2} \cline{4-5}
		$K_{5,2,1,1}$ & 20 & & $K_{3,2,2,2}$ & $\leq 372$\\ \cline{1-2} \cline{4-5}
		$K_{5,1,1,1,1}$ & 34 & & $K_{3,2,2,1,1}$ & $\leq 610$ \\ \cline{1-2} \cline{4-5}
		$K_{4,4,1}$ & 74 & & $K_{3,2,1,1,1,1}$ & $\leq 962$\\ \cline{1-2} \cline{4-5}
		$K_{4,3,2}$ & $\leq 120$ & & $K_{3,1,1,1,1,1,1}$ & $\leq 1432$ \\ \cline{1-2} \cline{4-5}
		$K_{4,3,1,1}$ & $\leq 164$ & & $K_{2,2,2,2,1}$ & $\leq 1098$ \\ \cline{1-2} \cline{4-5}
		$K_{4,2,2,1}$ & $\leq 178$ & & $K_{2,2,2,1,1,1}$ & $\leq 1576$ \\ \cline{1-2} \cline{4-5}
		$K_{4,2,1,1,1}$ & $\leq 244$ & & $K_{2,2,1,1,1,1,1}$ & $\leq 2139$\\ \cline{1-2} \cline{4-5}
		$K_{4,1,1,1,1,1}$ & $\leq 360$ & & $K_{2,1,1,1,1,1,1,1}$ & $\leq 2918$ \\ \cline{1-2} \cline{4-5}
		\multicolumn{3}{c |}{} & $K_9$ & $\leq 3987$ \\ \cline{4-5}
	\end{tabular}
	\end{center}
	\caption{Minimum number of links for complete partite graphs on 9 vertices} \label{Ta:9vertex}
\end{table}

Finally, in Section \ref{S:8vertex}, we find exact values or upper and lower bounds for the minimum number of knots in all intrinsically knotted complete partite graphs on 8 vertices.  The results are summarized in Table \ref{Ta:8vertex} (again, we only list the intrinsically knotted graphs; all others have knotless embeddings).

\begin{table} [htpb]
	\begin{center}
	\begin{tabular}{| l | c |}
		\hline
		$G$ & $mnk(G)$ \\ \hline
		$K_{3,3,1,1}$ & 1 \\ \hline                                                                                                                                                                                                                                                                                                                                                                                                                                          
		$K_{3,2,1,1,1}$ & 1\\ \hline
		$K_{3,1,1,1,1,1}$ & $3 \leq mnk \leq 4$\\ \hline
		$K_{2,2,1,1,1,1}$ & 2\\ \hline
		$K_{2,1,1,1,1,1,1}$ & $8 \leq mnk \leq 9$ \\ \hline
		$K_8$ & $15\leq mnk \leq 29$ \\ \hline
	\end{tabular}
	\end{center}
	\caption{Minimum number of knots for complete partite graphs on 8 vertices} \label{Ta:8vertex}
\end{table}

\section{Preliminaries and Notation} \label{S:prelim}

We begin by defining some useful notation and recalling some results from \cite{fm}.  Given a graph $G$ and a particular embedding $F$ of $G$ in $S^3$, a pair of disjoint cycles in $G$ is called {\it linked in F} if the corresponding embedded loops in $F$ form a non-trivial link.  Similarly, a cycle in $G$ is {\it knotted in F} if the corresponding embedded loop in $F$ is a non-trivial knot.  We let $nl(F)$ (respectively, $nk(F)$) denote the number of pairs of linked cycles (respectively, number of knotted cycles) in $F$.  Then the {\em minimum number of links (resp., knots) in $G$}, denoted $mnl(G)$ (resp., $mnk(G)$), is the minimum value of $nl(F)$ (resp., $nk(F)$) among all embeddings of $G$ in $S^3$.  $F$ is a {\em minimal link (resp. knot) embedding} of $G$ if $nl(F) = mnl(G)$ (resp., $nk(F) = mnk(G)$).  

An {\em $(m,n)$-link} in an embedding of a graph is a link of an $m$-cycle and an $n$-cycle.  We will often refer to 3-cycles as \emph{triangles}, 4-cycles as \emph{squares}, 5-cycles as \emph{pentagons}, etc.; this is purely for convenience and does not imply that the embedded cycles are regular polygons.  We will primarily detect links using the pairwise linking number.  We will say that a two-component link is \emph{odd} if the linking number is odd, and \emph{even} if the linking number is even.

As we are dealing with complete partite graphs, we will often describe the graphs (and their subgraphs) by indicating how the vertices are partitioned.  For example, the graph $K_{3,3,1}$ may be denoted $(abc)(123)(x)$; with $(ab)(12)(x)$ denoting a subgraph isomorphic to $K_{2,2,1}$.  Cycles in a graph will be denoted using square brackets, so $[a1x]$ would denote the 3-cycle with vertices $a$, $1$ and $x$.

Given loops $C$ and $D$ in $S^3$ such that $C \cap D$ is connected (or empty), we will define $C + D = \overline{(C\cup D) - (C\cap D)}$.  The notation is motivated by the observation that, given a cycle $S$ disjoint from $C$ and $D$, $\lk(S, C+D) = \lk(S,C) + \lk(S,D)$, where $\lk$ denotes the pairwise linking number.

Propositions \ref{P:unlinked}-\ref{P:K44} and Lemma \ref{L:FMCorrected} were proved by Fleming and Mellor \cite{fm}.  The statement of Lemma \ref{L:FMCorrected} in \cite{fm} contained a small error; here that error has been corrected by the addition of a sixth case (the error does not affect the validity of any other results in \cite{fm}).

\begin{prop} \label{P:unlinked}
For any $n$, the graphs $K_{n,1}$, $K_{n, 2}$, $K_{n,3}$, $K_{n, 1, 1}$, $K_{n, 2, 1}$ and $K_{n,1,1,1}$ have linkless embeddings.
\end{prop}

\begin{prop} \label{P:K331}
$mnl(K_{3,3,1}) = 1$.  Moreover, any embedding of $K_{3,3,1}$ contains an odd $(3,4)$-link.
\end{prop}

\begin{prop} \label{P:K3211}
$mnl(K_{3,2,1,1}) = 1$.  Moreover, any embedding of $K_{3,2,1,1}$ contains an odd $(3,4)$-link.
\end{prop}

\begin{prop} \label{P:K31111}
$mnl(K_{3,1,1,1,1}) = 3$.  Moreover, any embedding of $K_{3,1,1,1,1}$ contains at least 2 odd $(3,4)$-links and at least one odd $(3,3)$-link.
\end{prop}

\begin{prop} \label{P:K44}
$mnl(K_{4,4}) = 2$.  Moreover, any embedding of $K_{4,4}$ contains at least 2 odd $(4,4)$-links.
\end{prop}

\begin{lem} \label{L:FMCorrected}
Let $F$ be an embedding of $K_{2,2,1}$ (the 1-skeleton of a pyramid).  If a loop $C$ in $S^3$ has odd linking number with one of the faces of the pyramid in $F$, then it has odd linking number with at least 6 cycles in $F$.  Furthermore, $C$ is of one of the following six types:

\begin{enumerate}
\item $C$ has odd linking with 1 triangle, 3 squares and 3 pentagons in $F$, including $F$'s base square.
\item $C$ has odd linking with 2 triangles, 2 squares and 2 pentagons in $F$, {\em not} including $F$'s base square.
\item $C$ has odd linking with 2 triangles, 4 squares and 2 pentagons in $F$, {\em not} including $F$'s base square.
\item $C$ has odd linking with 4 triangles and 4 pentagons in $F$.
\item $C$ has odd linking with 3 triangles, 3 squares and 1 pentagon in $F$, including $F$'s base square.  Additionally, $C$ has even linking with a second pentagon in $F$.
\item $C$ has odd linking with 3 triangles, 5 squares, 1 pentagon in $F$, including $F$'s base square.  Additionally, $C$ has even linking with 1 triangle in $F$.
\end{enumerate} 
\end{lem}

\section{Some general results} \label{S:general}

In this section, we prove some general results for complete partite graphs where all but 4 of the vertices are in one partition -- i.e. for graphs $K_{n, 4}$, $K_{n, 3, 1}$, $K_{n, 2,2}$, $K_{n, 2, 1, 1}$ and $K_{n, 1, 1, 1, 1}$.  The results are summarized in Table \ref{Ta:general}.  The first of these graphs was dealt with by Fleming and Mellor \cite{fm}, who introduced the {\em fan embedding} for $K_{m,n}$.  The fan embedding for $K_{4,4}$ is shown in Figure \ref{F:fans}.

\begin{prop} \cite{fm} \label{P:Kn4}
$mnl(K_{n,4}) = 2 \binom{n}{4}$, and the minimum is realized by the fan embedding.
\end{prop}

We can get similar results for other graphs by using the fan embedding for $K_{n,4}$, together with (carefully chosen) additional edges among the four vertices in the second partition.  Figure \ref{F:fans} shows fan embeddings for $K_{4,4}$, $K_{4,3,1}$, $K_{3,3,1}$, $K_{3,2,1,1}$, $K_{4,2,2}$ and $K_{4,2,1,1}$.

    \begin{figure} [htpb]
    $$\scalebox{1}{\includegraphics{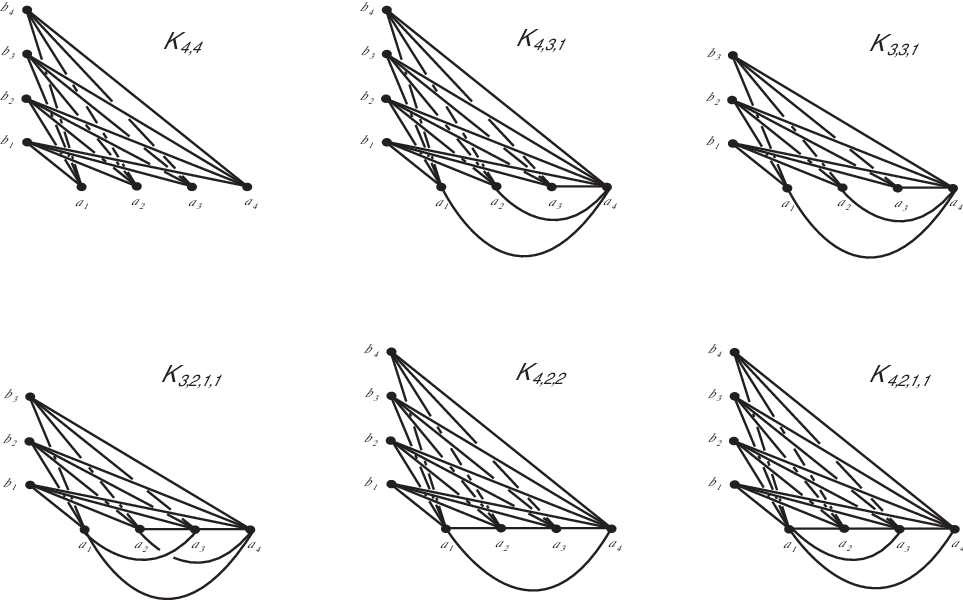}}$$
    \caption{Fan embeddings} \label{F:fans}
    \end{figure} 

\begin{prop} \label{P:Kn31}
$mnl(K_{n, 3, 1}) = {n \choose 3} + 2 {n \choose 4}$
\end{prop}
\begin{proof}
$K_{n,3,1} = (1\dots n)(abc)(x)$.  We first observe that the only possible pairs of linked cycles are $(3,4)$-links and $(4,4)$-links.  Since no cycle can contain adjacent vertices in $\{1, \dots, n\}$, and any cycle must use at least two vertices in $\{a,b,c,x\}$, two disjoint cycles must be either 4-cycles or 3-cycles.  Since any 3-cycle uses $x$, no two 3-cycles are disjoint.  So the only possible links are $(3,4)$-links and $(4,4)$-links.

$K_{n,3,1}$ contains $n \choose 3$ subgraphs isomorphic to $K_{3,3,1}$.  By Proposition \ref{P:K331}, each of these subgraphs contains at least one odd $(3,4)$-link.  Also, $K_{n,3,1}$ contains $n \choose 4$ subgraphs isomorphic to $K_{4,4}$, and by Proposition \ref{P:K44} each of these subgraphs contains at least two odd $(4,4)$-links.  All of these links are distinct, since each uses all the vertices in the respective subgraph.  So $mnl(K_{n,3,1}) \geq {n \choose 3} + 2 {n \choose 4}$.

However, in the fan embedding for $K_{n,3,1}$, the embedding of every subgraph isomorphic to $K_{3,3,1}$ is isotopic to the fan embedding of $K_{3,3,1}$ shown in Figure \ref{F:fans}.  This embedding has exactly one odd $(3,4)$-link, so the fan embedding of $K_{n,3,1}$ contains exactly $n \choose 3$ odd $(3,4)$-links.  Similarly, the embedding of every subgraph isomorphic to $K_{4,4}$ is isotopic to the fan embedding of $K_{4,4}$ shown in Figure \ref{F:fans}.  This embedding has exactly two odd $(4,4)$-links, so the fan embedding of $K_{n,3,1}$ contains exactly $2{n \choose 4}$ odd $(4,4)$-links.  Hence, the fan embedding is a minimal link embedding.
\end{proof}

\newpage
\begin{prop} \label{P:Kn22}
$mnl(K_{n,2,2}) = 2 {n \choose 4}$
\end{prop}
\begin{proof}
$K_{n,2,2}$ contains $n \choose 4$ subgraphs isomorphic to $K_{4,4}$, each of which contains at least two odd $(4,4)$-links by Proposition \ref{P:K44}.  So $mnl(K_{n,2,2}) \geq 2{n\choose 4}$.

As in Proposition \ref{P:Kn31}, all links in $K_{n,2,2}$ involve cycles of length at most four.  Any such link is contained in a subgraph isomorphic to $K_{4,2,2}$.  In the fan embedding of $K_{n,2,2}$, any such subgraph is isotopic to the fan embedding of $K_{4,2,2}$, which contains exactly two links, both odd $(4,4)$-links.  So the fan embedding of $K_{n,2,2}$ contains exactly $2{n\choose 4}$ links, and is a minimal link embedding.
\end{proof}

\begin{prop} \label{P:Kn211}
$mnl(K_{n,2,1,1}) = {n\choose 3} + 2{n \choose 4}$
\end{prop}
\begin{proof}
Since $K_{n,2,1,1}$ contains $K_{n,3,1}$ as a subgraph, $mnl(K_{n,2,1,1}) \geq \binom{n}{3} + 2\binom{n}{4}$.

In the fan embedding of $K_{n,2,1,1}$, any subgraph isomorphic to $K_{4,4}$ is isotopic to the fan embedding of $K_{4,4}$ shown in Figure \ref{F:fans}, which contains exactly two $(4,4)$-links.  Also any subgraph isomorphic to $K_{3,2,1,1}$ is isotopic to the fan embedding of $K_{3,2,1,1}$ shown in Figure \ref{F:fans}, which contains exactly one $(3,4)$-link.  Moreover, in $K_{n,2,1,1}$, any $(4,4)$-link is contained in a subgraph isomorphic to $K_{4,4}$ and any $(3,4)$-link is contained in a subgraph isomorphic to $K_{3,2,1,1}$.  Therefore, the fan embedding of $K_{n,2,1,1}$ contains exactly ${n\choose 3} + 2{n \choose 4}$ links, and is a minimal link embedding.
\end{proof}

For $K_{n,1,1,1,1}$ we need to modify our fan embedding -- we can't put all the edges among the last four vertices together, as we did for the other graphs in Figure \ref{F:fans}.  Instead, one of the edges needs to weave between the fans.  This is best shown using a different diagram for the fan embedding.  Figure \ref{F:Kn1111} shows the best embedding we have found for $K_{n,1,1,1,1}$.

    \begin{figure} [htpb]
    $$\scalebox{.5}{\includegraphics{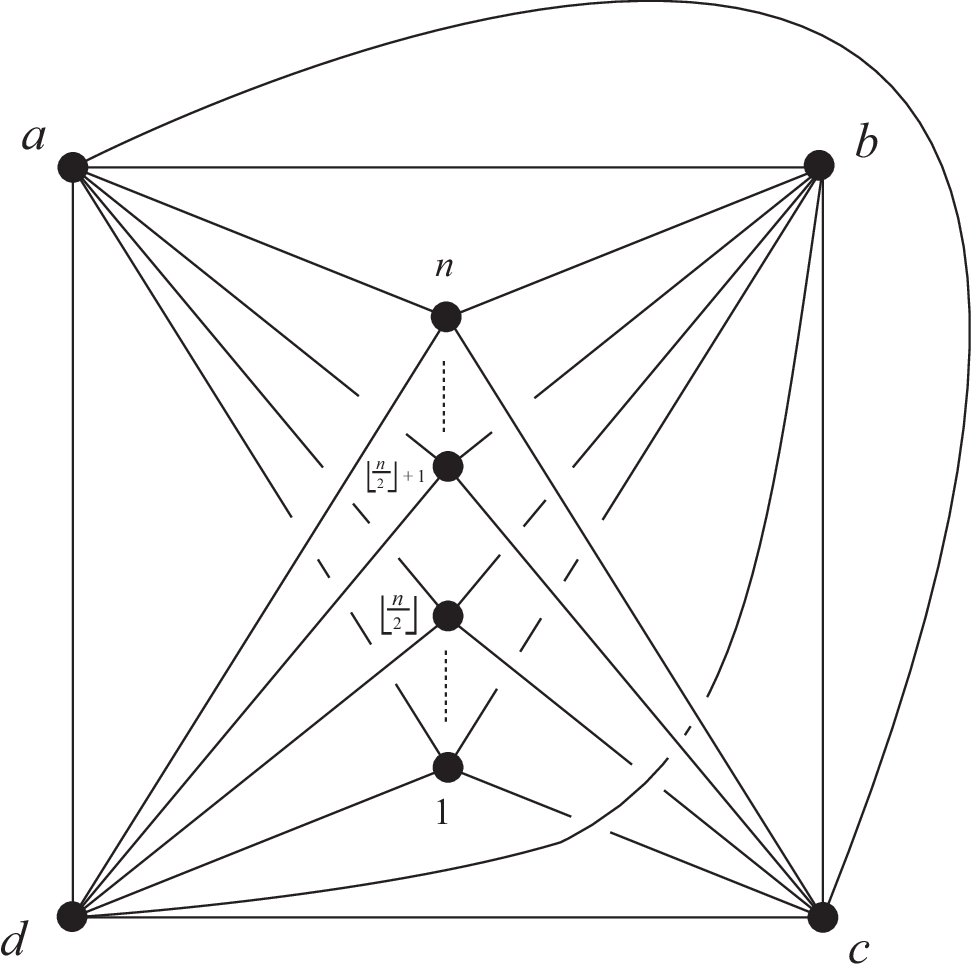}}$$
    \caption{An embedding of $K_{n,1,1,1,1}$} \label{F:Kn1111}
    \end{figure} 

\newpage
\begin{prop}
For $n > 2$, $2{n \choose 4} + 2{n \choose 3} + \left\lceil\frac{n^2-n}{6}\right\rceil \leq mnl(K_{n,1,1,1,1}) \leq  2{n \choose 4} + 2{n \choose 3} + \left\lceil\frac{n^2-2n}{4}\right\rceil$.
\end{prop}
\begin{proof}
We first prove the lower bound.  Any embedding $F = (1\dots n)(a)(b)(c)(d)$ of $K_{n,1,1,1,1}$ contains $n \choose 4$ subgraphs $(i_1i_2i_3i_4)(abcd)$ isomorphic to $K_{4,4}$.  So, by Proposition \ref{P:K44}, $F$ contains at least $2{n \choose 4}$ odd $(4,4)$-links.  Furthermore, $F$ contains $n\choose 3$ subgraphs isomorphic to $K_{3,1,1,1,1}$.  By Proposition \ref{P:K31111}, each of these subgraphs contains at least two odd $(3,4)$-links and one odd $(3,3)$-link.  This gives $2{n\choose 3}$ odd $(3,4)$-links and $n\choose 3$ odd $(3,3)$-links.  However, the $(3,3)$-links may not all be distinct; a given $(3,3)$-link uses only two of the vertices from $\{1,\dots,n\}$, so it will appear in $n-2$ different subgraphs isomorphic to $K_{3,1,1,1,1}$.  So there may be as few as $\frac{1}{n-2}{n \choose 3} = \frac{1}{n-2}\frac{n(n-1)(n-2)}{6} = \frac{n^2-n}{6}$ distinct odd $(3,3)$ links (since $n > 2$).  Since the number of links must be an integer, $F$ must contain at least $\left\lceil\frac{n^2-n}{6}\right\rceil$ odd $(3,3)$ links.  Adding up the three kinds of links gives the desired lower bound.

To prove the upper bound, we will describe an embedding of $K_{n,1,1,1,1}$ with this many links.  The embedding $F = (1\dots n)(a)(b)(c)(d)$ is shown in Figure \ref{F:Kn1111}.  If we remove the edge $\overline{bd}$, we get an embedding of $K_{n,2,1,1}$ which is isotopic to the fan embedding.  The edge $\overline{bd}$ is drawn in the ``middle" of the $n$ fans -- i.e. it crosses {\em over} edge $\overline{id}$ for $i \leq \lfloor\frac{n}{2}\rfloor$, and {\em under} edge $\overline{id}$ for $i \geq \lfloor\frac{n}{2}\rfloor+1$.  We will show that this embedding has $2{n\choose 4}$ odd $(4,4)$-links, $2{n\choose 3}$ odd $(3,4)$-links and $\left\lceil\frac{n^2-2n}{4}\right\rceil$ odd $(3,3)$-links, and no others.  Since at least half the vertices in any cycle must be selected from $\{a,b,c,d\}$, and any cycle must use at least two of these vertices, there are no links using cycles of length 5 or more.  We first consider the $(4,4)$-links.  The cycles in a $(4,4)$-link will not use any of the edges between the vertices $a,b,c,d$, so the number of $(4,4)$-links in $F$ is the number in the subgraph isomorphic to $K_{n,4}$.  But the embedding of this subgraph is isotopic to the fan embedding, and contains $2{n\choose 4}$ $(4,4)$-links (all odd), by Proposition \ref{P:Kn4}.

Every $(3,4)$-link is contained in a subgraph isomorphic to $K_{3,1,1,1,1}$.  Depending on the choice of the three independent vertices, every subgraph of $K_{3,1,1,1,1}$ in $F$ is isotopic to one of the embeddings $F_3^0$, $F_3^1$, $F_3^2$ or $F_3^3$ shown in Figure \ref{F:K31111}.  All of these embeddings have exactly two $(3,4)$-links (both odd); since each $(3,4)$-link uses all seven vertices in $K_{3,1,1,1,1}$, the links from different subgraphs are distinct.  So $F$ contains exactly $2{n\choose 3}$ $(3,4)$-links (all odd).

    \begin{figure} [htpb]
    $$\scalebox{.9}{\includegraphics{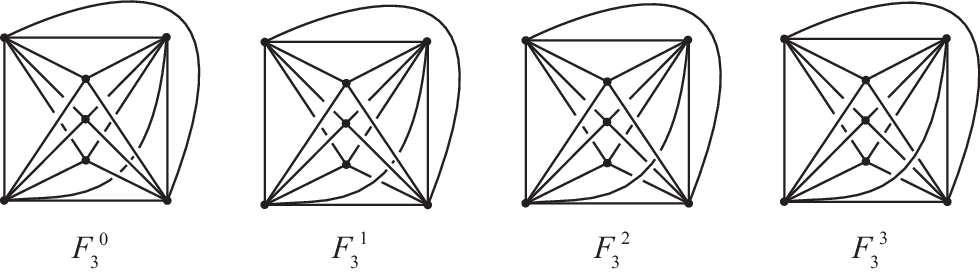}}$$
    \caption{Embeddings of $K_{3,1,1,1,1}$} \label{F:K31111}
    \end{figure} 

Finally, every $(3,3)$-link is contained in a subgraph isomorphic to $K_{2,1,1,1,1}$.  Depending on the choice of the two independent vertices, every subgraph of $K_{2,1,1,1,1}$ in $F$ is isotopic to one of the embeddings $F_2^0$, $F_2^1$, or $F_2^2$ shown in Figure \ref{F:K21111}.  $F_2^0$ and $F_2^2$ each contain one odd $(3,3)$-link, while $F_2^1$ contains no links.  So the number of $(3,3)$-links is equal to the number of embedded subgraphs isotopic to $F_2^0$ or $F_2^2$, which is the number of ways of choosing two vertices $i,j$ so that either $i,j \leq \frac{n}{2}$ or $i,j > \frac{n}{2}$.  There are two cases, depending on whether $n$ is odd or even.  If $n = 2k$ is even, then the number of choices is ${k \choose 2} + {k\choose 2} = 2{k \choose 2} = k(k-1) = k^2-k = \frac{n^2}{4} - \frac{n}{2} = \frac{n^2-2n}{4}$.  This is an integer, so it is also equal to $\left\lceil\frac{n^2-2n}{4}\right\rceil$.  On the other hand, if $n = 2k+1$ is odd, then the number of choices is ${k \choose 2} + {{k+1} \choose 2} = \frac{k(k-1)}{2} + \frac{(k+1)k}{2} = \frac{k(k-1+k+1)}{2} = \frac{k(2k)}{2} = k^2 = \left(\frac{n-1}{2}\right)^2 = \frac{n^2-2n+1}{4} = \frac{n^2-2n}{4} + \frac{1}{4}$.  Since this is an integer which is only $\frac{1}{4}$ more than $\frac{n^2-2n}{4}$, it must be $\left\lceil\frac{n^2-2n}{4}\right\rceil$.  Therefore, in either case, the number of $(3,3)$-links is $\left\lceil\frac{n^2-2n}{4}\right\rceil$.  Adding this to the number of $(4,4)$- and $(3,4)$-links gives the desired upper bound.
\end{proof}

    \begin{figure} [htpb]
    $$\scalebox{.9}{\includegraphics{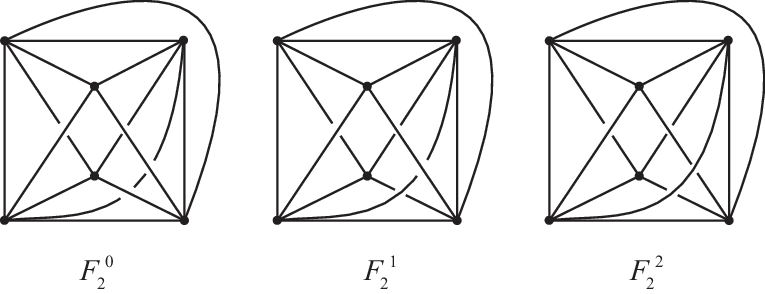}}$$
    \caption{Embeddings of $K_{2,1,1,1,1}$} \label{F:K21111}
    \end{figure} 

It is worth observing that the difference between the upper and lower bounds is $O(n^2)$, while the bounds themselves are $O(n^4)$, so the difference is relatively small compared to the bounds.  Table \ref{Ta:bounds} shows how they compare for $n \leq 12$.  In particular, the bounds agree for $n=5$, so $mnl(K_{5,1,1,1,1}) = 34$.

\begin{table} [htpb]
\begin{center}
\begin{tabular}{c|c|c|c|c|c|c|c|c|c|c}
	$n$ & 3 & 4 & 5 & 6 & 7 & 8 & 9 & 10 & 11 & 12 \\ \hline
	$2{n \choose 4} + 2{n \choose 3} + \left\lceil\frac{n^2-n}{6}\right\rceil$ & 3 & 12 & 34 & 75 & 147 & 262 & 432 & 675 & 1009 & 1452 \\ \hline
	$2{n \choose 4} + 2{n \choose 3} + \left\lceil\frac{n^2-2n}{4}\right\rceil$ & 3 & 12 & 34 & 76 & 149 & 264 & 436 & 680 & 1015 & 1460 \\
\end{tabular}
\end{center}
\caption{Upper and lower bounds for $mnl(K_{n,1,1,1,1})$} \label{Ta:bounds}
\end{table}

\section{Complete partite graphs with 9 vertices} \label{S:9vertex}

Now that we have dealt with all complete partite graphs where all but four vertices are in one partition, we consider more complex complete partite graphs.  Fleming and the second author \cite{fm} considered complete partite graphs with 8 or fewer vertices; we will consider complete partite graphs with 9 vertices.  Our results are summarized in Table \ref{Ta:9vertex}.

We do not list the graphs which have unlinked embeddings by Proposition \ref{P:unlinked} (in fact, since these graphs do not contain pairs of disjoint cycles, any embedding is unlinked).  The values for the first five graphs ($K_{5,4}$ through $K_{5,1,1,1,1}$) follow from the results of Section \ref{S:general}, when $n = 5$.  Determining $K_{4,4,1}$ is significantly more difficult, and is the topic of Section \ref{SS:K441}.  It quickly becomes clear that the subsequent graphs will require even more elaborate arguments, so for all the graphs after $K_{4,4,1}$, we have only determined upper bounds for $mnl(G)$.  Appendix \ref{S:9vertexdiagrams} provides embeddings which realize the minimum number of links (where known), or the upper bound given in Table \ref{Ta:9vertex}.  The embedding for $K_9$ is based on the minimal crossing diagram presented by Guy \cite{gu}.
    
\begin{rem}
Our methods only provide lower bounds on the minimum number of links {\it with non-zero linking number.}  For the exact values of $mnl(G)$ provided in Table \ref{Ta:9vertex}, we have checked that the embeddings in Appendix \ref{S:9vertexdiagrams} do not contain any non-trivial links with trivial linking number by using a computer to list possible candidates for such links and then checking them by hand.  For the subsequent graphs, there are far more candidates.  In the future, we hope to refine our program to reduce the number of possibilities to a size that can be checked by hand.  In the meantime, the upper bound listed in Table \ref{Ta:9vertex} is really only an upper bound on the minimum number of links with non-zero linking number.
\end{rem}

\subsection{Minimum number of links for $K_{4,4,1}$} \label{SS:K441}

In this section, we will show that $mnl(K_{4,4,1}) = 74$.  We begin with a lemma that may be useful for many graphs; this is a variation on a lemma proved by Johnson and Johnson \cite{jj}, and is proved similarly.

\begin{lem} \label{L:K33}
Let $F$ be an embedding of $K_{3,3}$, and let $C$ be a loop in $S^3$ disjoint from $F$ which links at least 1 cycle in $F$ with odd linking number.  Then $C$ must link exactly 8 cycles in $F$ with odd linking number.  Furthermore, $C$ must link either 4 squares and 4 hexagons, or 6 squares and 2 hexagons in $F$.
\end{lem}
\begin{proof}
Let $F$ be an embedding of $K_{3,3}$ with $\{v_1,v_2,v_3\}$ and $\{w_1, w_2, w_3\}$ denoting its two sets of independent vertices.  Orient the edges $\{\overline{v_iw_j} \mid {i,j} \in \{1,2,3\}\}$ from $v_i$ to $w_j$, and denote each oriented edge by $\overrightarrow{v_iw_j}$; we also choose an orientation for $C$.  Consider a diagram for $C \cup F$ (i.e. a projection to $S^2$ where the edges are in general position, and at each crossing we record which edge crosses over the other).  Let $c_{i,j}$ be the number of crossings between $C$ and $\overrightarrow{v_iw_j}$, counted with sign.  $F$ contains 15 cycles: 9 squares and 6 hexagons.  For each cycle in $F$, we can write its linking number with $C$ as a sum of the $c_{i,j}$'s.  Let $s_k = \lk(C,S_k)$ for each square $S_k$ in $F$, and $h_l = \lk(C,H_l)$ for each hexagon $H_l$.  Then:

\begin{align*}
2s_1 &= 2\lk(C,[v_1w_1v_2w_2]) = c_{1,1} - c_{2,1} + c_{2,2} - c_{1,2} \\
2s_2 &= 2\lk(C,[v_1w_1v_2w_3]) = c_{1,1} - c_{2,1} + c_{2,3} - c_{1,3} \\
2s_3 &= 2\lk(C,[v_1w_1v_3w_2]) = c_{1,1} - c_{3,1} + c_{3,2} - c_{1,2} \\
2s_4 &= 2\lk(C,[v_1w_1v_3w_3]) = c_{1,1} - c_{3,1} + c_{3,3} - c_{1,3} \\
2s_5 &= 2\lk(C,[v_1w_2v_2w_3]) = c_{1,2} - c_{2,2} + c_{2,3} - c_{1,3} \\
2s_6 &= 2\lk(C,[v_1w_2v_3w_3]) = c_{1,2} - c_{3,2} + c_{3,3} - c_{1,3} \\
2s_7 &= 2\lk(C,[v_2w_1v_3w_2]) = c_{2,1} - c_{3,1} + c_{3,2} - c_{2,2} \\
2s_8 &= 2\lk(C,[v_2w_1v_3w_3]) = c_{2,1} - c_{3,1} + c_{3,3} - c_{2,3} \\
2s_9 &= 2\lk(C,[v_2w_2v_3w_3]) = c_{2,2} - c_{3,2} + c_{3,3} - c_{2,3} \\
2h_1 &= 2\lk(C,[v_1w_1v_2w_2v_3w_3]) = c_{1,1} - c_{2,1} + c_{2,2} - c_{3,2} + c_{3,3} - c_{1,3} \\
2h_5 &= 2\lk(C,[v_1w_3v_2w_1v_3w_2]) = c_{1,3} - c_{2,3} + c_{2,1} - c_{3,1} + c_{3,2} - c_{1,2} \\
2h_2 &= 2\lk(C,[v_1w_1v_2w_3v_3w_2]) = c_{1,1} - c_{2,1} + c_{2,3} - c_{3,3} + c_{3,2} - c_{1,2} \\
2h_4 &= 2\lk(C,[v_1w_2v_2w_3v_3w_1]) = c_{1,2} - c_{2,2} + c_{2,3} - c_{3,3} + c_{3,1} - c_{1,1} \\
2h_3 &= 2\lk(C,[v_1w_2v_2w_1v_3w_3]) = c_{1,2} - c_{2,2} + c_{2,1} - c_{3,1} + c_{3,3} - c_{1,3} \\
2h_6 &= 2\lk(C,[v_1w_3v_2w_2v_3w_1]) = c_{1,3} - c_{2,3} + c_{2,2} - c_{3,2} + c_{3,1} - c_{1,1} 
\end{align*}

We can eliminate the variables $c_{i,j}$ to write $s_1,\dots, s_9$ and $h_1, h_2$ in terms of $h_3, h_4, h_5$ and $h_6$.

\begin{align}
3s_1 &= -2h_3 - 2h_4 - h_5 - h_6 \label{E.s1} \\
3s_2 &= -h_3 - h_4 - 2h_5 - 2h_6 \label{E.s2} \\
3s_3 &= -h_3 - h_4 + h_5 - 2h_6 \label{E.s3} \\
3s_4 &= h_3 - 2h_4 - h_5 - h_6 \label{E.s4} \\
3s_5 &= h_3 + h_4 - h_5 - h_6 \label{E.s5} \\
3s_6 &= 2h_3 - h_4 - 2h_5 + h_6 \label{E.s6} \\
3s_7 &= h_3 + h_4 + 2h_5 - h_6 \label{E.s7} \\
3s_8 &= 2h_3 - h_4 + h_5 + h_6 \label{E.s8} \\
3s_9 &= h_3 - 2h_4 - h_5 + 2h_6 \label{E.s9} \\
h_1 &= -h_4 - h_5 \label{E.h1} \\
h_2 &= -h_3 - h_6 \label{E.h2} 
\end{align}

We first observe that if all the $h_i$'s are even, so are all the $s_j$'s.  Thus, if $C$ has odd linking with a square in $F$, it must also have odd linking with at least one hexagon in $F$.

The converse is also true.  Observe from the crossing equations that every $h_i = \lk(C,H_i)$ is the sum or difference of two $s_j$'s.  As an example, consider $h_1$:

\begin{align*}
h_1 &= \frac{1}{2}(c_{1,1} - c_{2,1} + c_{2,2} - c_{3,2} + c_{2,3} - c_{1,3}) \\
&= \frac{1}{2}(c_{1,1} - c_{2,1} + c_{2,2} - c_{1,2} + c_{1,2} - c_{3,2} + c_{2,3} - c_{1,3}) \\
&= \frac{1}{2}(2s_1 + 2s_2) \\
&= s_1 + s_2
\end{align*}

Thus, if $h_i$ is odd, some $s_j$ (where $S_j$ shares three edges with $H_i$) must also be odd.  So $C$ must have odd linking with both a square and a hexagon which share three edges.  Without loss of generality, assume that $C$ links $S_1$ and $H_1$ with odd linking number, so $s_1$ and $h_1$ are odd.  With $h_1$ odd, \eqref{E.h1} requires that exactly one of $h_4$ and $h_5$ is odd.  By \eqref{E.h2} either none of $h_2$, $h_3$ or $h_6$ is odd, or exactly 2 of them are odd.  This leaves us with 8 cases:

{\sc Case 1:} $h_1$ and $h_4$ are odd and $h_2$, $h_3$, $h_5$ and $h_6$ are even.  Since $s_1$ is odd, this contradicts equation \eqref{E.s1}.

{\sc Case 2:} $h_1$ and $h_5$ are odd and $h_2$, $h_3$, $h_4$ and $h_6$ are even.  Then $s_1$, $s_3$, $s_4$, $s_5$, $s_8$ and $s_9$ must also be odd.

{\sc Case 3:} $h_1$, $h_2$, $h_3$ and $h_4$ are odd, $h_5$ and $h_6$ are even.  This contradicts equation \eqref{E.s1}.

{\sc Case 4:} $h_1$, $h_2$, $h_4$ and $h_6$ are odd, $h_3$ and $h_5$ are even.  Then, $s_1$, $s_2$, $s_3$ and $s_4$ must also be odd.

{\sc Case 5:} $h_1$, $h_2$, $h_3$ and $h_5$ are odd, $h_4$ and $h_6$ are even.  Then, $s_1$, $s_2$, $s_7$ and $s_8$ must also be odd.

{\sc Case 6:} $h_1$, $h_2$, $h_5$ and $h_6$ are odd, $h_3$ and $h_4$ are even.  This contradicts equation \eqref{E.s1}.

{\sc Case 7:} $h_1$, $h_3$, $h_4$ and $h_6$ are odd, $h_2$ and $h_5$ are even.  Then, $s_1$, $s_5$, $s_7$ and $s_9$ must also be odd.

{\sc Case 8:} $h_1$, $h_3$, $h_5$ and $h_6$ are odd, $h_2$ and $h_4$ are even.  This contradicts equation \eqref{E.s1}.

Therefore, $C$ has odd linking with either 2 hexagons and 6 squares (in Case 2), or 4 hexagons and 4 squares (in Cases 4, 5, 7).
\end{proof} 

\begin{prop}
$mnl(K_{4,4,1}) = 74$
\end{prop}
\begin{proof}
We first observe that the embedding of $K_{4,4,1}$ in Appendix \ref{S:9vertexdiagrams} contains 74 links, so $mnl(K_{4,4,1}) \leq 74$.  Let $F=(abcd)(1234)(x)$ be a minimum link embedding of $K_{4,4,1}$, so $nl(F) \leq 74$.  We need to show that $nl(F) = 74$.

We first observe that $K_{4,4,1}$ contains $\binom{4}{3}\binom{4}{3} = 16$ subgraphs isomorphic to $K_{3,3,1}$, and so $F$ contains at least 16 odd $(3,4)$-links by Proposition \ref{P:K331}.  Also, $K_{4,4,1}$ contains $1 + \binom{4}{3} + \binom{4}{3} = 9$ subgraphs isomorphic to $K_{4,4}$, and so $F$ contains at least $2*9 = 18$ odd $(4,4)$-links by Proposition \ref{P:K44}.  We will be particularly interested in the subgraph $G = (abcd)(1234)$.  Ultimately, we will show that, if $F$ is a minimum link embedding, then $G$ must contain exactly two odd $(4,4)$-links.

\begin{claim} \label{C:4triangles}
$F$ must have at least 4 distinct triangles in odd $(3,4)$-links.
\end{claim}
\begin{proof}
As we saw above, every subgraph of $F$ isomorphic to $K_{3,3,1}$ must contain an odd $(3,4)$-link.  Without loss of generality, consider the subgraph $(abc)(123)(x)$ and let $[xa1]$ be a triangle in an odd $(3,4)$-link.  Now consider the subgraph $(bcd)(234)(x)$.  This subgraph must also contain a triangle in an odd $(3,4)$-link, which is not $[xa1]$.  Without loss of generality let this triangle be $[xb2]$.  Next consider the subgraph $L = (x)(acd)(234)$, which contains neither $[xa1]$ nor $[xb2]$.  There are two cases:

\medskip

\noindent {\sc Case 1:} The linked triangle in $L$ contains either $a$ or $2$, or both.  Now consider the subgraph $(bcd)(134)(x)$.  This subgraph contains a triangle in an odd $(3,4)$-link which contains neither $a$ nor $2$, and so is distinct from the previous three linked triangles.  Hence $F$ has at least 4 distinct triangles in odd $(3,4)$-links.

\medskip

\noindent{\sc Case 2:} The linked triangle in $L$ does not involve either of the vertices $a$ or $2$.  Without loss of generality, let this triangle be $[xc3]$.  Towards contradiction, assume that $[xa1]$, $[xb2]$ and $[xc3]$ are the only three triangles in odd $(3,4)$-links in $F$.  Consider a square in $F$ that links $[xa1]$.  Without loss of generality, let this square be $[b2c3]$.  As $[xa4]$ and $[xd1]$ are not in odd $(3,4)$-links, $[b2c3]$ also must have odd linking with $[x4a1] = [xa1] + [xa4]$ and $[xa1d] = [xa1] + [xd1]$.

By this argument, each of the 16 odd $(3,4)$-links in $F$ induces two $(4,4)$-links.  These links are all distinct, since they involve different sets of vertices, so $F$ contains at least 32 odd $(4,4)$-links.  However, all of these 32 links involve the vertex $x$; in addition, the subgraph $G = (abcd)(1234)$ contains at least two odd $(4,4)$-links, for a total of 34 odd $(4,4)$-links.

Observe that the complement of a triangle in $F$ is a subgraph isomorphic to $K_{3,3}$.  Hence, by Lemma \ref{L:K33}, each triangle involved in an odd $(3,4)$-link is also in at least two odd $(3,6)$-links.  Thus, our three triangles give us at least 6 $(3,6)$-links.

Now we want to count the number of squares involved in odd $(3,4)$-links.  Since we have at least 16 odd $(3,4)$-links, and only 3 different triangles, one of the triangles must link at least $\lceil{\frac{16}{3}}\rceil = 6$ different squares.  Without loss of generality, say that $[xa1]$ links $k \geq 6$ squares.  Since the only square in the complement of $[xa1]$ which does not involve either $b$ or $2$ is $[c3d4]$, at most one of the $k$ squares could also link $[xb2]$; similarly, at most one could also link $[xc3]$.  Without loss of generality, $[xb2]$ is involved in at least half of the remaining odd $(3,4)$-links, and at most one of the squares that link $[xb2]$ also links $[xa1]$.  This means there are at least $k + \lceil{\frac{16-k}{2}}\rceil - 1$ different squares.  Since $k \geq 6$, the total number of squares is at least
$$k + \left\lceil{\frac{16-k}{2}}\right\rceil - 1 \geq k + \frac{16}{2} - \frac{k}{2} - 1 = 7 + \frac{k}{2} \geq 10$$
So there are at least 10 different squares in odd $(3,4)$-links.

The complement of a square in $F$ that does not involve $x$ is a subgraph isomorphic to $K_{2,2,1}$.  By Lemma \ref{L:FMCorrected}, if a square has odd linking with a triangle, it must have odd linking with two pentagons (or odd linking with one pentagon and non-zero even linking with a pentagon or triangle).  So each square in an odd $(3,4)$-link gives us two new links, for a total of at least 20 additional links.  So the total number of links in $F$ is at least $16 + 34 + 6 + 20 = 76$.  But we know that $nl(F) \leq 74$, so this is a contradiction.  Therefore, $F$ must have at least 4 distinct triangles in odd $(3,4)$-links.
\end{proof}

Now we will consider the odd $(4,4)$-links in $G = (abcd)(1234)$ (considered as a subset of $F$).  Since a square in one of these links does not contain the vertex $x$, its complement in $F$ is isomorphic to $K_{2,2,1}$.  Moreover, since it links a square which does not contain $x$, it must be of type 1, 5, or 6, as described in Lemma \ref{L:FMCorrected}.

\begin{claim} \label{C:3square}
If $G$ has at least 3 odd $(4,4)$-links, then it cannot have a $(4,4)$-link with both squares of type 5 or 6.
\end{claim}
\begin{proof}
Assume that $G$ has at least 3 odd $(4,4)$-links and that both squares in one of the links, denoted $S$, are of type 5 or 6.  Each of the squares in $S$ must then have odd linking with at least 3 triangles in $F$, for a total of $6$ odd $(3,4)$-links.  Observe that each linked triangle must contain the vertex $x$ and an edge from the other square in $S$.  Hence, all 6 triangles must be distinct.  By Lemma \ref{L:K33}, each of these 6 triangles must oddly link a total of 8 squares or hexagons in $F$, yielding $48$ odd $(3,4)$- or $(3,6)$-links.

Since each of the 6 squares in the 3 square-square links in $G$ is of type 1, 5 or 6, they must link at least 2 other cycles (one oddly linked pentagon and either a second oddly linked pentagon or an evenly linked pentagon or triangle) for $12$ additional links.  Since $F$ must contain at least 18 odd $(4,4)$-links, we get a total of at least $48 + 12 + 18 = 78$ links in $F$.  But we know that $nl(F) \leq 74$, which gives the desired contradiction.
\end{proof}

\begin{claim} \label{C:5square}
$G$ must have fewer than 5 odd $(4,4)$-links.
\end{claim}
\begin{proof}
Assume to the contrary, that $G$ has at least 5 odd $(4,4)$-links.  By Claim \ref{C:4triangles}, $F$ has at least 4 triangles involved in odd $(3,4)$-links.  Since the complement of each triangle in $F$ is a subgraph isomorphic to $K_{3,3}$, Lemma \ref{L:K33} implies that there are at least $4 * 8 = 32$ odd $(3,4)$- or $(3,6)$-links in $F$.  As we mentioned above, the 10 squares in the 5 $(4,4)$-links in $G$ must be of type 1, 5, or 6.  So each of these squares is in at least 3 odd $(4,4)$-links, one odd $(4,5)$-link, and at least one other link (linking either a pentagon, or a triangle with even linking number).  This gives 30 $(4,4)$-links and 20 other links, but the 5 links in $G$ are counted twice.  Thus, there must be at least $32 + (30-5) + 20 = 77$ links in $F$.  But $nl(F) \leq 74$, so this is a contradiction.
\end{proof}

\begin{claim} \label{C:5triangles}
If $G$ has at least two odd $(4,4)$-links where one square in each link is type 1 and the other is type 5 or 6, then $F$ must have at least 5 distinct triangles in odd $(3,4)$-links.
\end{claim}
\begin{proof}
Assume $G$ has at least two odd $(4,4)$-links where one square in each link is type 1 and the other is type 5 or 6.  Without loss of generality, one of the $(4,4)$-links is $[a1b2]/[c3d4]$ where $[a1b2]$ is type 1 and $[c3d4]$ is type 5 or 6.  Then $[c3d4]$ must link at least 3 triangles.  Without loss of generality, let these three triangles be $[xa1]$, $[xa2]$ and $[xb1]$.  Square $[a1b2]$ must link at least one triangle.  Without loss of generality, let this triangle be $[xc3]$.

If in fact there are only 4 distinct triangles in $(3,4)$-links in $F$ (the minimum required by Claim \ref{C:4triangles}), the other $(4,4)$-link in $G$ with one square of type 1 and the other of type 5 or 6 must induce 4 $(3,4)$-links linking the same 4 triangles.  However, only $[c3d4]$ can link all of the triangles $[xa1]$, $[xa2]$ and $[xb1]$.  Thus, the second square of type 1 must link exactly 1 of $[xa1]$, $[xa2]$ and $[xb1]$, while the second square of type 5 or 6 must link $[xc3]$ and the two remaining triangles.  But, no subgraph isomorphic to $K_{2,2,1}$ can contain $[xc3]$ and two of $[xa1]$, $[xa2]$ and $[xb1]$, as this would require 6 vertices.  Thus $F$ must have at least 5 distinct triangles in odd $(3,4)$-links.
\end{proof}

\begin{claim} \label{C:type234}
If $G$ has $m$ squares of type 1 and $n$ squares of type 5 or 6, and $m + 3n < 16$, then $F$ contains at least $16 - (m+3n)$ odd $(4,5)$-links where the square is type 2, 3, or 4.
\end{claim}
\begin{proof}
We know that $F$ contains at least 16 odd $(3,4)$ links.  Since each square of type 1 is in one odd $(3,4)$-link, and each square of type 5 or 6 is in 3 odd $(3,4)$-links, the total number of odd $(3,4)$-links where the square is type 1, 5 or 6 is $m + 3n$.  If $m + 3n < 16$, then there must be additional odd $(3,4)$-links where the square is type 2, 3, or 4.  If $p$ is the number of squares in $G$ of type 2, $q$ is the number of type 3 and $r$ is the number of type $4$, then Lemma \ref{L:FMCorrected} implies there are $2p+2q+4r$ additional odd $(3,4)$-links, so $2p+2q+4r \geq 16-(m+3n)$.  There are also, by Lemma \ref{L:FMCorrected}, $2p+2q+4r$ additional odd $(4,5)$-links.  So there are at least $16-(m+3n)$ odd $(4,5)$-links where the square is type 2, 3, or 4.
\end{proof}

We have shown that $G$ must have fewer than 5 odd $(4,4)$-links and, if it has 3 or more odd $(4,4)$ links, it cannot have both squares of one of these links of type 5 or 6.  Our next claim is that $G$ in fact must have exactly 2 odd $(4,4)$-links.

\begin{claim} \label{C:2squares}
$G$ must have exactly 2 odd $(4,4)$-links.
\end{claim}
\begin{proof}
By Proposition \ref{P:K44} and Claim \ref{C:5square}, $G$ has either 2, 3, or 4 odd $(4,4)$-links.  We first consider the possibilities when $G$ has 4 odd $(4,4)$-links, and show that each one leads to a contradiction.  By Claim \ref{C:3square}, none of the $(4,4)$-links in $G$ can have both squares of type 5 or 6.  So in each link either both squares are of type 1, or one is type 1 and the other is type 5 or 6.  If two or more of the links have one square of type 1 and the other of type 5 or 6, then, by Claim \ref{C:5triangles}, $F$ has at least 5 distinct triangles involved in odd $(3,4)$-links; by Lemma \ref{L:K33}, this means there are at least $5*8 = 40$ odd $(3,4)$- and $(3,6)$-links.  Otherwise, there are at least 4 distinct triangles (by Claim \ref{C:4triangles}), and hence at least 32 odd $(3,4)$- and $(3,6)$-links.  Since squares of types 1, 5 and 6 all link three other squares, there are $8*3 - 4 = 20$ distinct odd $(4,4)$-links (we need to subtract 4 because the $(4,4)$-links in $G$ are counted twice).  Finally, the squares of type 1 yield 3 additional links with pentagons, and those of type 5 or 6 give two additional links (either with pentagons, or even links with triangles).  In each case, we can add these all up to get a minimum number of links in $F$, as shown in Table \ref{Ta:4links}; in every case we find that this minimum is larger than 74, which contradicts the fact that $nl(F) \leq 74$.

\begin{table}[htpb]
\begin{center}
\begin{tabular}{c|c|c|c|c|c}
	1/1 links & 1/(5, 6) links & $(3,*)$-links & $(4,4)$-links & other links & total \\ \hline
	0 & 4 & 40 & 20 & $4*3 + 4*2 = 20$ & $80 > 74$ \\
	1 & 3 & 40 & 20 & $5*3 + 3*2 = 21$ & $81 > 74$ \\
	2 & 2 & 40 & 20 & $6*3 + 2*2 = 22$ & $82 > 74$ \\
	3 & 1 & 32 & 20 & $7*3 + 1*2 = 23$ & $75 > 74$ \\
	4 & 0 & 32 & 20 & $8*3 = 24$ & $76 > 74$ \\
\end{tabular}
\end{center}
\caption{Cases when $G$ has 4 odd $(4,4)$-links.} \label{Ta:4links}
\end{table}

We now consider the possibilities when $G$ has 3 odd $(4,4)$-links, and again show that each leads to a contradiction.  In addition to the calculations we made for the case of 4 links, and to recalling that $F$ has at least 18 odd $(4,4)$-links, we have one additional observation.  If there are $m$ squares of type 1 and $n$ squares of type 5 or 6, there are at least $16-(m+3n)$ odd $(4,5)$-links where the square is type 2, 3, or 4, by Claim \ref{C:type234}.  Adding these odd $(4,5)$-links to the total forces $F$ to have more than 74 links, giving our contradiction.  The various cases are summarized in Table \ref{Ta:3links}.

\begin{table} [htpb]
\begin{center}
\begin{tabular}{c|c|c|c|c|c|c}
	1/1 links & 1/(5, 6) links & $(3,*)$-links & $(4,4)$-links & other 1,5,6 links & other 2,3,4 links & total \\ \hline
	0 & 3 & 40 & 18 & $3*3 + 3*2 = 15$ & $16 - (3 + 3*3) = 4$ & $77 > 74$ \\
	1 & 2 & 40 & 18 & $4*3 + 2*2 = 16$ & $16 - (4 + 2*3) = 6$ & $80 > 74$ \\
	2 & 1 & 32 & 18 & $5*3 + 1*2 = 17$ & $16 - (5 + 1*3) = 8$ & $75 > 74$ \\
	3 & 0 & 32 & 18 & $6*3 = 18$ & $16 - 6 = 10$ & $78 > 74$ \\
\end{tabular}
\end{center}
\caption{Cases when $G$ has 3 odd $(4,4)$-links.} \label{Ta:3links}
\end{table}

Since we have ruled out any possibility that $G$ has 3 or 4 odd $(4,4)$-links, $G$ must have exactly 2 odd $(4,4)$-links.
\end{proof}

Now that we know that $G$ has exactly 2 odd $(4,4)$-links, we ask whether the four squares in these links are type 1, 5, or 6.  We will find that, for $F$ to be a minimal link embedding, they must all be of type 1.  We need to consider six cases.

\medskip

\noindent {\sc Case 1:} Assume all four squares are type 5 or 6.  By Lemma \ref{L:FMCorrected}, this means each square is in three odd $(3,4)$-links.  Since the two squares in a single $(4,4)$-link must link distinct triangles, there must be at least 6 distinct triangles in odd $(3,4)$-links.  Thus, by Lemma \ref{L:K33}, $F$ has at least $6 * 8 = 48$ odd $(3,4)$- and $(3,6)$-links.  Each square of type 5 or 6 also links two other cycles (either pentagons, or triangles with even linking number), giving an additional $2 * 4 = 8$ links.  Since we know $F$ contains at least $18$ odd $(4,4)$-links we get a total of at least $48 + 8 + 18 = 74$ links.  However, by Claim \ref{C:type234}, there are also at least $16-3*4 = 4$ more odd $(4,5)$-links with squares of type 2, 3, or 4, for a total of at least 78 links.  Since $nl(F) \leq 74$, this is a contradiction.

\medskip

\noindent {\sc Case 2:} Assume one square is type 1, and the other three are type 5 or 6.  As in the previous case, there are at least $48$ odd $(3,4)$- and $(3,6)$-links.  Each square of type 5 or 6 also links two other cycles (either pentagons, or triangles with even linking number), giving an additional $2 * 3 = 6$ links.  From the square linking case 1 we get 3 odd $(4,5)$-links.  Since $F$ must contain at least 18 odd $(4,4)$-links, we get at least $48 + 18 + 6 + 3 = 75$ links.  Since $nl(F) \leq 74$, this is a contradiction.

\medskip

\noindent{\sc Case 3:} Assume one $(4,4)$-link has both squares of type 1, and the other has both squares of type 5 or 6.  As in Case 2, $F$ has at least 48 odd $(3,4)$- and $(3,6)$-links, 18 odd $(4,4)$-links, $2 * 3 = 6$ odd $(4,5)$-links from the squares of type 1, and $2 * 2 = 4$ other links from the squares of type 5 or 6.  This gives a total of $48 + 18 + 6 + 4 = 76$ links.  Since $nl(F) \leq 74$, this is a contradiction.

\medskip

\noindent{\sc Case 4:} Assume both $(4,4)$-links have one square of type 1 and the other of type 5 or case 6.  By Claim \ref{C:5triangles}, $F$ contains at least 5 distinct triangles in odd $(3,4)$-links.  So, by Lemma \ref{L:K33}, $F$ contains at least $5 * 8 = 40$ odd $(3,4)$- and $(3,6)$-links.  As in the previous cases, there are $2*3 = 6$ odd $(4,5)$ links using the squares of type 1, and $2*2 = 4$ other links using the squares of type 5 or 6.  We also know $F$ contains at least 18 odd $(4,4)$-links, giving $40 + 6 + 4 + 18 = 68$ links.  Moreover, by Claim \ref{C:type234} there are also at least $16 - (2 + 3*2) = 8$ odd $(4,5)$-links using squares of type 2, 3, or 4.  This gives a total of at least $68 + 8 = 76$ total links.  Since $nl(F) \leq 74$, this is a contradiction.

\medskip

\noindent {\sc Case 5:} Assume one square is type 5 or 6, and the other three are type 1.  By Lemma \ref{L:FMCorrected}, these squares are involved in at least $3*3 + 1$ odd $(4,5)$-links, plus one even $(4,5)$- or $(3,4)$-link for a total of $11$ links.  By Claim \ref{C:type234}, there must also be at least $16 - (3 + 3) = 10$ odd $(4,5)$-links using squares of type 2, 3, or 4.  Let $p$, $q$, and $r$ be the number of squares of types 2, 3, and 4, respectively.  Then, as in Claim \ref{C:type234}, these squares are in at least $2p + 2q + 4r \geq 10$ odd $(3,4)$-links, and the same number of odd $(4,5)$-links.  They are also in at least $2p + 4q$ odd $(4,4)$-links, by Lemma \ref{L:FMCorrected}. 

We now claim that $r$ must be 0.  Towards contradiction, assume $r \geq 1$.  Then there is a square $S$ of type 4, which links at least 4 triangles.  Without loss of generality, let $[c3d4]$ be the square of type 5 or 6, and $[a1b2]$ the square of type 1 which links it.  Then $[a1b2]$ links one triangle with odd linking number -- without loss of generality, triangle $[xc3]$.  $[c3d4]$ links three triangles -- without loss of generality, $[xa1]$, $[xa2]$ and $[xb1]$.  For $S$ to link the same 4 triangles, the triangles would all need to be in the complement of the square in $F$.  However, the four triangles involve 7 different vertices, while the complement of $S$ has only 5 vertices.  Thus, there must be at least 5 distinct triangles involved in odd $(3,4)$-links.  By Lemma \ref{L:K33}, this gives at least $5*8=40$ $(3,4)$- or $(3,6)$-links.  In addition, $F$ contains at least 18 odd $(4,4)$-links.  Together with the 20 odd $(4,5)$-links, and one even $(4,5)$- or $(3,4)$-link, counted previously, $F$ contains at least $40+18+21=79$ links.  Since $nl(F) \leq 74$, this is a contradiction.

With $r=0$, there must be $2p + 2q \geq 10$ odd $(3,4)$-links, $2p + 2q \geq 10$ odd $(4,5)$-links and $2p + 4q \geq 10 + 2q \geq 10$ odd $(4,4)$ links using squares of types 2, 3 or 4.  By Claim \ref{C:4triangles}, $F$ contains at least 4 triangles involved in odd $(3,4)$-links, and so by Lemma \ref{L:K33}, $F$ contains at least $4*8=32$ odd $(3,4)$- and $(3,6)$-links.  The four squares of types 1, 5 or 6 each link 3 squares with odd linking number, giving $3*4=12$ odd $(4,4)$-links.  So there are a total of at least 22 odd $(4,4)$-links in $F$.  Together with the 20 odd $(4,5)$-links, and one even $(4,5)$- or $(3,4)$-link, counted previously, $F$ contains at least $32+22+21=75$ total links.  Since $nl(F) \leq 74$, this is a contradiction.

\medskip

\noindent{\sc Case 6:} Assume all four squares are type 1.  By Claim \ref{C:4triangles} and Lemma \ref{L:K33}, there are at least $4*8=32$ odd $(3,4)$- and $(3,6)$-links in $F$.  We also know that $F$ contains at least 18 odd $(4,4)$-links.  By Lemma \ref{L:FMCorrected}, the four squares of type 1 each link 3 pentagons, giving $3*4=12$ odd $(4,5)$-links.  Also, by Claim \ref{C:type234}, there are another $16 - 4 = 12$ odd $(4,5)$-links using squares of type 2, 3, or 4.  This means that $F$ contains at least $32+18+12+12=74$ links.

\medskip

Therefore, $nl(F) \geq 74$.  Since we already know that $nl(F) \leq 74$, we conclude that $nl(F) = 74$, and therefore $mnl(K_{4,4,1}) = 74$.  Moreover, the minimal case occurs only when $G$ contains exactly two odd $(4,4)$-links, with all four squares of type 1.
\end{proof}

\section{Counting knots in complete partite graphs on 8 vertices} \label{S:8vertex}

We now turn to counting the minimal number of knots in a graph, rather than links.  Our results are summarized in Table \ref{Ta:8vertex}.  We only list the complete partite graphs on 8 vertices which are intrinsically knotted, as determined by Blain et. al. \cite{bb2}.  Appendix \ref{S:8vertexdiagrams} shows embeddings realizing the upper bounds in Table \ref{Ta:8vertex}.  It is worth observing that these embeddings also realize the known upper bounds for the minimum number of links (see \cite{fm}), which leads us to pose the following question for future investigation:

\begin{question}
Does every graph have an embedding which simultaneously realizes the minimum linking number and the minimum knotting number?
\end{question}

We have identified knots using the second coefficient of the Conway polynomial.  There are no non-trivial knots with fewer than 8 crossings whose Conway polynomial has a non-zero second coefficient (see \cite{knotinfo}), and in the embeddings shown in Appendix \ref{S:8vertexdiagrams}, there are no cycles with more than 7 self-crossings.  So the second coefficient of the Conway polynomial is sufficient to identify all knotted cycles in these embeddings.

As with links, we establish lower bounds on the minimum number of knots by looking for subgraphs with known numbers of knotted cycles.  We begin with two well-known minor-minimal intrinsically knotted graphs:  $K_7$ and $K_{3,3,1,1}$.  The first result is due to Conway and Gordon \cite{cg}.

\begin{prop} \label{P:k7} \cite{cg}
Every embedding of $K_7$ contains at least one knotted 7-cycle.  Moreover, there is an embedding of $K_7$ with exactly one knotted cycle.
\end{prop}

Motwani, Raghunathan and Saran \cite{mrs} showed that $K_7$ is minor-minimal among intrinsically knotted graphs, so no other graph on 6 or 7 vertices is intrinsically knotted.  The only minor-minimal intrinsically knotted complete partite graph on 8 vertices is $K_{3,3,1,1}$, which Foisy \cite{fo} proved was intrinsically knotted.  Kohara and Suzuki \cite{ks} found an embedding of $K_{3,3,1,1}$ with exactly one trefoil knot (another example of such an embedding is shown in Appendix \ref{S:8vertexdiagrams}).  Thus, we obtain:

\begin{prop} \label{P:k3311}
$mnk(K_{3,3,1,1}) = 1$.
\end{prop}

Foisy's proof does not provide much information as to the length of the knotted cycle.  Foisy and Ludwig \cite{folu} have asked whether every embedding of $K_{3,3,1,1}$ contains a knotted Hamiltonian cycle.  Every example we have found supports the conjecture that the answer to this questions is ``Yes", but we are not able to prove it.  As a result, subgraphs isomorphic to $K_{3,3,1,1}$ are not as useful in counting knotted cycles, since it becomes hard to prove that the counted cycles are distinct.  Fortunately, we can make use of another intrinsically knotted graph.

Motwani, Raghunathan and Saran \cite{mrs} also showed that $\triangle$-Y moves preserve intrinsic knottedness, where a $\triangle$-Y move removes the edges of a 3-cycle, and adds a new vertex adjacent to the vertices of the original 3-cycle, as shown in Figure \ref{F:triangleY}.

    \begin{figure} [htpb]
    $$\scalebox{.3}{\includegraphics{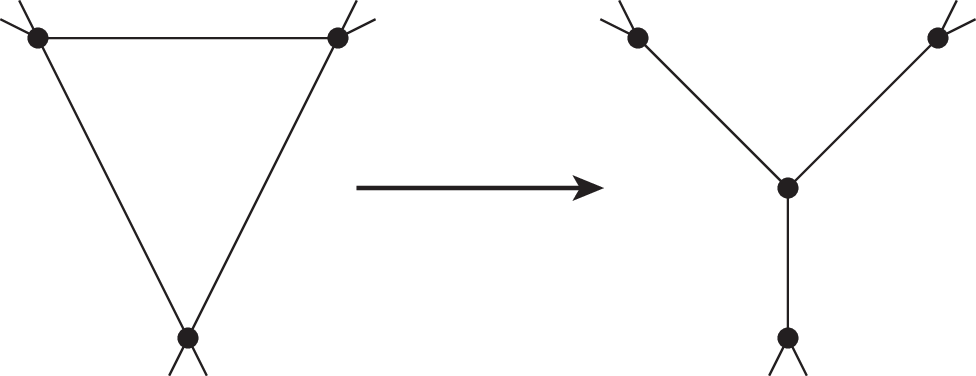}}$$
    \caption{$\triangle$-Y move} \label{F:triangleY}
    \end{figure} 

In particular, the result of performing a $\triangle$-Y move on $K_7$ is the 8-vertex graph $H_8$, shown in Figure \ref{F:H8}.  We will denote the vertices of $H_8$ as $(v)(abc)(1)(2)(3)(4)$, where $v$ has valence 3 (added by the $\triangle$-Y move), $a,b,c$ are three mutually non-adjacent vertices, all adjacent to vertex $v$, with valence 5, and $1,2,3,4$ are four vertices with valence 6 (adjacent to all vertices except $v$).  We will call the vertex of valence 3 the {\em top} vertex, the vertices of valence 5 the {\em middle} vertices, and the vertices of valence 6 the {\em bottom} vertices.  The following lemma has also been proved independently by Nikkuni and Taniyama \cite{nt}, as a corollary to a stronger result about graphs related to $K_7$ by $\triangle$-Y moves.

    \begin{figure} [htpb]
    $$\scalebox{.75}{\includegraphics{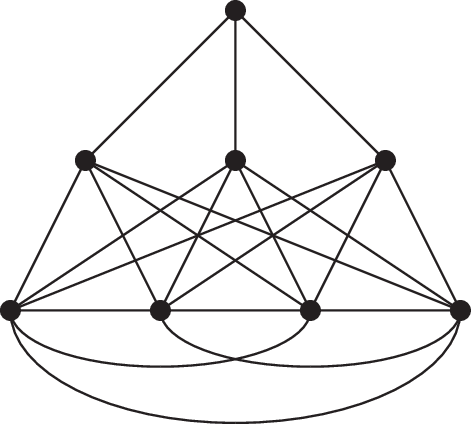}}$$
    \caption{The graph $H_8$} \label{F:H8}
    \end{figure} 
    
\begin{lem} \label{L:h8}
Every embedding of $H_8$ contains either a knotted 8-cycle or a knotted 7-cycle which contains all the bottom vertices.
\end{lem}

\begin{proof} 
Let $\Gamma$ be an embedding of $H_8$, and let $v$ denote the top vertex.  Then there is an embedding $\Gamma'$ of $K_7$ which differs from $\Gamma$ only in a neighborhood of the edges adjacent to $v$, as shown in Figure \ref{F:H8K7}.

    \begin{figure} [htpb]
    $$\scalebox{.75}{\includegraphics{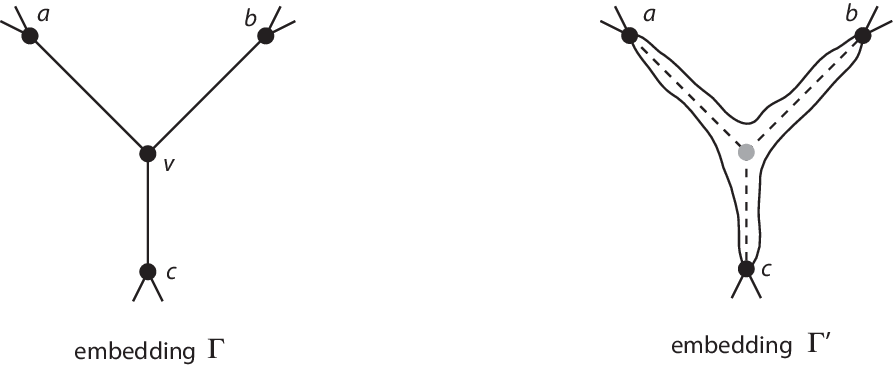}}$$
    \caption{Embeddings $\Gamma$ of $H_8$ and $\Gamma'$ of $K_7$} \label{F:H8K7}
    \end{figure} 

By Proposition \ref{P:k7}, $\Gamma'$ contains a knotted 7-cycle $C$.  If $C$ does not contain the edges $ab$, $ac$ or $bc$ in Figure \ref{F:H8K7}, then it is also a knotted 7-cycle in $\Gamma$ that does not contain $v$, and so contains all the bottom vertices.  If $C$ does contain one of these three edges, say $ab$, then there is a corresponding knotted 8-cycle in $\Gamma$ obtained by replacing $ab$ with $av$ and $vb$.  If $C$ contains two of the three edges, say $ab$ and $bc$, then $C$ is isotopic to the embedded cycle obtained by replacing these two edges with $av$ and $vc$ (since the triangle $\triangle abc$ in $\Gamma'$ is null-homotopic in the complement of the graph).  In this case, we obtain a knotted 7-cycle which contains $v$, but does not contain one of the vertices adjacent to $v$, and so again contains all the bottom vertices.
\end{proof}

So we will count knotted cycles by looking for subgraphs isomorphic to either $K_7$ or $H_8$.  The fact that we know the length of at least one knotted cycle in these subgraphs, by Proposition \ref{P:k7} and Lemma \ref{L:h8}, gives us much more power in counting knotted cycles.  We will call the knotted cycle of length 7 or 8 required by Lemma \ref{L:h8} the {\bf required} knot in a graph isomorphic $H_8$.

\begin{prop} \label{P:k32111}
$mnk(K_{3,2,1,1,1}) = 1$
\end{prop}

\begin{proof}
Partition the vertices of $K_{3,2,1,1,1}$ as $(abc)(xy)(1)(2)(3)$.  There are two subgraphs isomorphic to $H_8$, formed by taking either $x$ or $y$ as the top vertex.  Since $H_8$ is intrinsically knotted, so is $K_{3,2,1,1,1}$.  The embedding of $K_{3,2,1,1,1}$ in Appendix \ref{S:8vertexdiagrams} has exactly one knotted cycle, so $mnk(K_{3,2,1,1,1}) = 1$.
\end{proof}

\begin{prop} \label{P:k221111}
$mnk(K_{2,2,1,1,1,1}) = 2$
\end{prop}

\begin{proof}
Partition the vertices of $K_{2,2,1,1,1,1}$ as $(ab)(xy)(1)(2)(3)(4)$, where $a$ and $b$ are not adjacent, and $x$ and $y$ are not adjacent.  Then there are 16 subgraphs isomorphic to $H_8$:  the top vertex can be any of $a,b,x,y$, and once that choice is made, we choose one of $1,2,3,4$ to be the third middle vertex.  For example, one such subgraph is $(a)(xy1)(b)(2)(3)(4)$ (where $a$ is the top vertex).

Each of these subgraphs contains a knotted 7-cycle containing all four bottom vertices, or a knotted 8-cycle.  If a 7-cycle $C$ is the required knot in one subgraph, then whichever vertex is missed by $C$ is a bottom vertex in a different subgraph, and must be part of the required knot in that subgraph.  Therefore, a single knotted 7-cycle cannot be the required knot in all 16 subgraphs.  Also, a knotted 8-cycle cannot appear in all 16 subgraphs.  To see this, observe that one of $a,b,x,y$ must be adjacent to one of $1,2,3,4$ in the 8-cycle.  Without loss of generality, assume that $a$ is adjacent to $1$.  But then this cycle does not appear in the subgraph $(x)(ab1)(y)(2)(3)(4)$ (where $x$ is the top vertex).  So no 8-cycle can appear in all 16 subgraphs.

Hence, $K_{2,2,1,1,1,1}$ must contain at least two knotted cycles.  The embedding in Appendix \ref{S:8vertexdiagrams} contains exactly two knotted cycles, so $mnl(K_{2,2,1,1,1,1}) = 2$.
\end{proof}

\begin{prop} \label{P:k311111}
$3 \leq mnk(K_{3,1,1,1,1,1}) \leq 4$
\end{prop}

\begin{proof}
Partition the vertices of $K_{3,1,1,1,1,1}$ as $(abc)(1)(2)(3)(4)(5)$.  There are 5 subgraphs isomorphic to $H_8$, depending on which of $1,2,3,4,5$ is chosen to be the top vertex (in each case, $\{a,b,c\}$ are the middle vertices).  Each of these subgraphs contains a knotted 7-cycle (containing all four bottom vertices) or a knotted 8-cycle.  

A knotted 7-cycle can be the required knot in only one of the subgraphs.  To show this, we consider two cases.  First, let $C$ be a 7-cycle containing all three of $a, b, c$.  Then $C$ can only be the required knot in the single subgraph where it does not contain the top vertex.  In the second case, assume $C$ contains two of $a, b, c$, say $a$ and $b$.  Then there must be a vertex $v$ adjacent to both $a$ and $b$ in $C$ (or $C$ is not in any of the subgraphs), and $C$ only appears in the subgraph where $V$ is the top vertex.  In either case, $C$ can be the required knot in only one of the subgraphs.

A knotted 8-cycle can appear in at most two of the subgraphs.  To see this, suppose an 8-cycle appears in three of the subgraphs.  In each subgraph, the top vertex is adjacent to two of $a,b,c$.  Since the three subgraphs have different top vertices, this would mean the 8-cycle has three vertices, each adjacent to two of $a,b,c$.  But this forces a 6-cycle, which is a contradiction.

So a given knotted cycle can be the only knotted cycle in at most 2 of the 5 subgraphs, which means there are at least $\left\lceil \frac{5}{2} \right\rceil = 3$ knotted cycles.  The embedding shown in Appendix \ref{S:8vertexdiagrams} has 4 knotted cycles.  Hence $3 \leq mnk(K_{3,1,1,1,1,1}) \leq 4$.
\end{proof}

\begin{prop} \label{P:k2111111}
$8 \leq mnk(K_{2,1,1,1,1,1,1}) \leq 9$
\end{prop}

\begin{proof}
Partition the vertices of $K_{2,1,1,1,1,1,1}$ as $(ab)(1)(2)(3)(4)(5)(6)$.  There are two subgraphs isomorphic to $K_7$ (formed by taking one of $a$ or $b$, with the other 6 vertices), so there are at least two knotted 7-cycles.

There are 70 subgraphs isomorphic to $H_8$, split into two types.  Type 1 subgraphs are formed by taking one of $1,2,3,4,5,6$ as the top vertex, and grouping another of these vertices with $a$ and $b$ as the middle vertices.  There are $6 \cdot 5 = 30$ subgraphs of Type 1.  Type 2 subgraphs are formed by taking one of $a$ or $b$ as the top vertex, and grouping three of $1,2,3,4,5,6$ as the middle vertices.  There are $2 \binom{6}{3} = 2 \cdot 20 = 40$ subgraphs of Type 2.  So there are a total of $30 + 40 = 70$ subgraphs isomorphic to $H_8$.

Each of these subgraphs contains a knotted 7-cycle that contains all four bottom vertices, or a knotted 8-cycle.  We first consider the knotted 7-cycles.  Let $C$ be a knotted 7-cycle; we first consider the case when $C$ contains both $a$ and $b$.  Since $a$ and $b$ cannot be adjacent, they are separated by either 1 or 2 vertices along $C$.  Without loss of generality, $C$ is either $[a2b3456]$ or $[a23b456]$.  If $C = [a2b3456]$, then it appears in 5 of the subgraphs:  two subgraphs where 1 is the top vertex, and 4 or 5 is chosen as the third middle vertex; one subgraph where 1 is the third middle vertex, and 2 is the top vertex (the top vertex must be adjacent to both $a$ and $b$); one where $a$ is the top vertex and 2, 6 and 1 are the middle vertices; and similarly one where $b$ is the top vertex.  If $C = [a23b456]$, then it appears in only three subgraphs:  one where 1 is the top vertex and $a, b, 5$ are the middle vertices, one where $a$ is the top vertex, and one where $b$ is the top vertex.

Now we consider the case when $C$ contains only one of $a$ and $b$.  Without loss of generality, say that $C = [a123456]$.  Then $C$ appears in 2 subgraphs where $a$ and $b$ are middle vertices, and the top vertex is either 1 or 6 (the third middle vertex is either 2 or 5, respectively).  $C$ also appears in 4 subgraphs where $b$ is the top vertex, $a$ is a bottom vertex, and the middle vertices are three of 1, 2, 3, 4, 5, 6 which are non-adjacent in $C$ ($\{1,3,5\}$, $\{1,3,6\}$, $\{1,4,6\}$ or $\{2,4,6\}$).  So in this case $C$ appears in at most 6 subgraphs isomorphic to $H_8$.

Now let $C$ be a knotted 8-cycle in a subgraph isomorphic to $H_8$.  Then $C$ contains both $a$ and $b$, but they are not adjacent in the cycle.  So $a$ and $b$ are separated by one, two or three vertices in $C$.  Without loss of generality, $C = [a1b23456], [a12b3456]$, or $[a123b456]$.  It is not hard to check that in the first case $C$ appears in 11 of the subgraphs isomorphic to $H_8$, and in the latter two cases $C$ appears in 8 of the subgraphs (the details are left as an exercise for the reader).

The two 7-cycles coming from the $K_7$ subgraphs each appear in at most 6 of the $H_8$ subgraphs.  The remaining $70 - 2*6 = 58$ $H_8$ subgraphs each contain a knotted cycle, but each cycle can be the required knot in at most 11 of the subgraphs.  So there are at least $\left\lceil \frac{58}{11} \right\rceil = 6$ different knotted cycles, besides the two arising from the $K_7$ subgraphs.  So there are at least 8 knotted cycles in $K_{2,1,1,1,1,1,1}$.  The embedding in Appendix \ref{S:8vertexdiagrams} has exactly 9 knotted cycles, so $mnk(K_{2,1,1,1,1,1,1}) = 8$ or $9$.
\end{proof}

\begin{prop} \label{P:k8}
$15 \leq mnk(K_8) \leq 29$
\end{prop}

\begin{proof}
$K_8$ has 8 subgraphs isomorphic to $K_7$, so any embedding contains at least 8 knotted 7-cycles.  There are also $8\binom{7}{3} = 280$ subgraphs isomorphic to $H_8$.  A given 7-cycle can be the required knot in 14 of the $H_8$ subgraphs:  either the top vertex is the vertex not in the 7-cycle, and there are 7 choices of three vertices which are mutually non-adjacent in the 7-cycle as the middle vertices; or one of middle vertices is the vertex not in the 7-cycle, the top vertex is any of the 7 vertices in the cycle, and the other middle vertices are the vertices in the cycle adjacent to the top vertex.  A given 8-cycle appears in 24 of the $H_8$ subgraphs (8 choices for the top vertex, and 3 choices for the three mutually non-adjacent vertices which are adjacent to the top vertex).  So the 8 knotted 7-cycles from the $K_7$ subgraphs can account for the required knotted cycles in at most $14\cdot 8 = 112$ of the $H_8$ subgraphs, leaving $280 - 112 = 168$ other $H_8$ subgraphs.  A given knotted cycle can account for at most 24 of these subgraphs, so there are at least $\left\lceil \frac{168}{24}\right\rceil = 7$ additional knotted cycles.  So an embedding of $K_8$ contains at least 15 knotted cycles.  The embedding in Appendix \ref{S:8vertexdiagrams} contains 29 knotted cycles (8 knotted 7-cycles and 21 knotted 8-cycles).
\end{proof}


\newpage
\appendix

\section{Minimum linking number embeddings of complete partite graphs with 9 vertices} \label{S:9vertexdiagrams}

    \begin{figure} [htpb]
    $$\scalebox{1}{\includegraphics{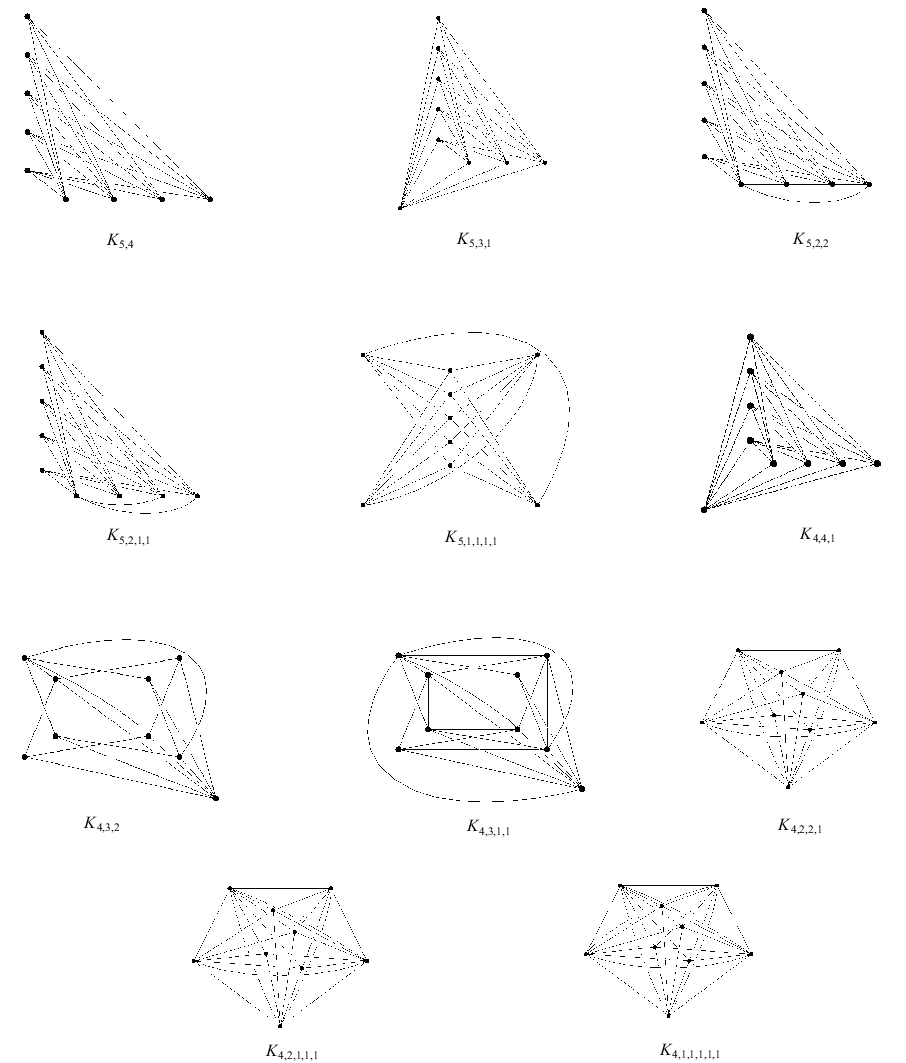}}$$
    \caption{Complete Partite Graphs with 9 Vertices} \label{F:9vertex1}
    \end{figure} 

    \begin{figure} [htpb]
    $$\scalebox{.9}{\includegraphics{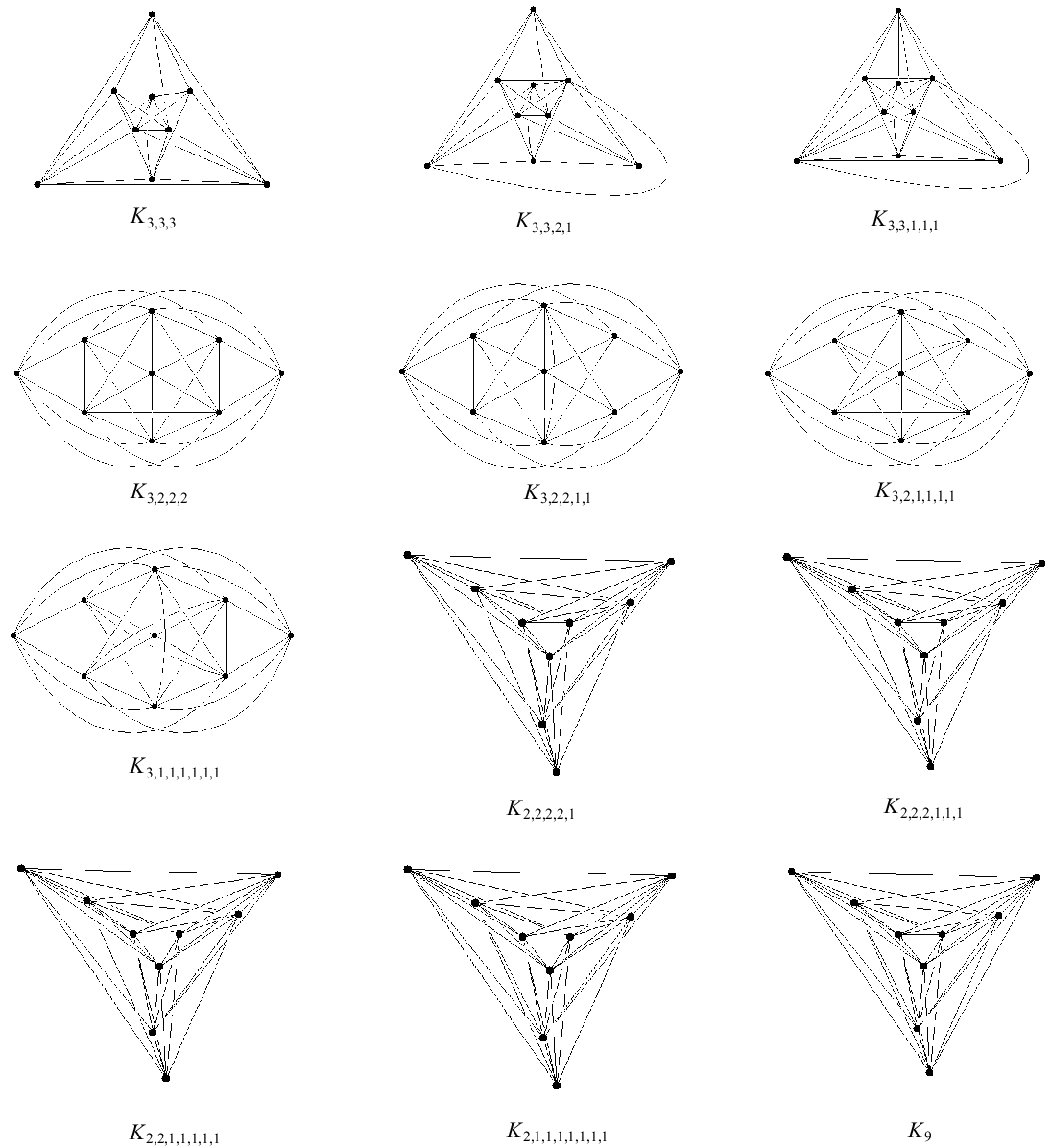}}$$
    \caption{More Complete Partite Graphs with 9 Vertices} \label{F:9vertex2}
    \end{figure} 

\section{Minimum knotting number embeddings of complete partite graphs with 8 vertices} \label{S:8vertexdiagrams}

    \begin{figure} [htpb]
    $$\scalebox{.8}{\includegraphics{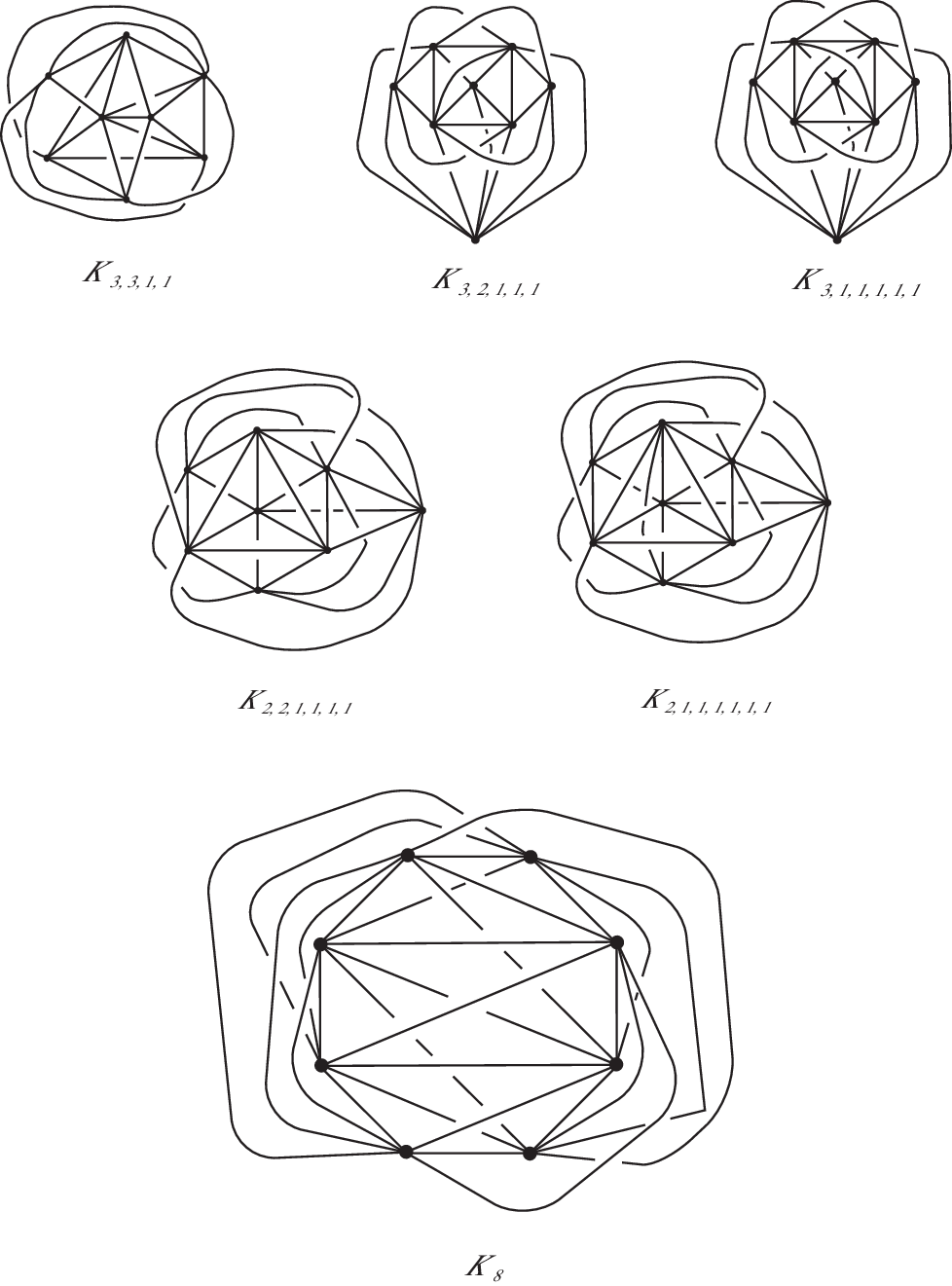}}$$
    \caption{Intrinsically knotted complete partite graphs on 8 vertices} \label{F:8vertex}
    \end{figure} 
    
\section{Procedures for computing the number of links and knots in a spatial graph} \label{S:code}

In this section we will describe the algorithms used to count the number of links and knots in a spatial graph.  The full program is available at {\it http://myweb.lmu.edu/bmellor/research/Gordian}.

\subsection{Representing the graph and finding the cycles} \label{SS:represent}

The first task is to represent a particular graph embedding in a form usable by the program.  The user supplies the number of vertices in the graph, the edges of the graph, and information about each crossing in the graph (the details of the input format are described in the user's manual at {\it http://myweb.lmu.edu/bmellor/research/Gordian}).  The vertices of the graph are numbered from $0$ to $n-1$ (for a graph with $n$ vertices), and the abstract graph is stored as an adjacency matrix (i.e. an array) \texttt{int[n][n] adjacent}, where \texttt{adjacent[i][j]} contains the value 1 if vertices \texttt{i} and \texttt{j} are connected by an edge, and the value 0 otherwise.

We also need to record the details of the particular spatial embedding of the graph; specifically, how the edges are crossing each other.  We assume each edge is oriented from its smaller endpoint to its larger endpoint (i.e. if \texttt{i} and \texttt{j} are adjacent vertices with $i < j$, then the edge is oriented from \texttt{i} to \texttt{j}).  If there are $k$ crossings, then the crossing information is contained in another array \texttt{int[k][7] crossings}.  For crossing \texttt{i} we have: \begin{itemize}
	\item \texttt{crossings[i][0]} and \texttt{crossings[i][1]} are the endpoints of the edge that crosses {\it over} the other edge.
	\item \texttt{crossings[i][2]} and \texttt{crossings[i][3]} are the endpoints of the edge that crosses {\it under} the other edge.
	\item \texttt{crossings[i][4]} gives the order of the crossing among all the crossings along the over-crossing edge (following the orientation of the edge).
	\item \texttt{crossings[i][5]} gives the order of the crossing among all the crossings along the under-crossing edge.
	\item \texttt{crossings[i][6]} gives the sign of the crossing ($+1$ or $-1$).
\end{itemize}

We are interested in counting the knotted and linked cycles in the graph, so the next step is to generate all the cycles in the graph.  This is done by going through all possible cycles (i.e. all cycles in the complete graph on the given vertices), and then removing any cycles which contain an edge not contained in the adjacency matrix for the graph.  Each cycle is given an (arbitrary) orientation determined by the order of the vertices; this may not agree with the default orientation on each edge.

\subsection{Counting linked cycles} \label{SS:link}

To count the linked cycles, the program computes the linking number for each pair of disjoint cycles.  This requires counting the crossings between each pair of edges (with sign).  For efficiency, the program first extracts this data from the \texttt{crossings} array, constructing an array \texttt{int[n][n][n][n] crossingMatrix}, where \texttt{crossingMatrix[a][b][c][d]} is the sign of the crossing between edges $(a,b)$ and $(c,d)$, oriented from $a$ to $b$ and from $c$ to $d$.  So \texttt{crossingMatrix[b][a][c][d]} is the negative of \texttt{crossingMatrix[a][b][c][d]}, and so forth.

The program then compares each cycle to every other cycle, and follows the following procedure for each pair of cycles: \begin{enumerate}
	\item If the cycles share any vertices, then the pair is not disjoint; move on to the next pair.
	\item Compare each edge of the first cycle to each edge of the second cycle, and add up their crossing numbers (from \texttt{crossingMatrix}); use the orientation for each edge induced by the orientation of the cycle.
	\item The sum is the linking number.  If the linking number is nonzero, add this pair to the list of links, and increase the number of links by 1.
	\item Move on to the next pair of cycles.
\end{enumerate}

The result is a list of all pairs of cycles with non-trivial linking number, and the number of such pairs.

\subsection{Counting knotted cycles} \label{SS:knot}

The program identifies knotted cycles using the second coefficient of the Conway polynomial, $a_2$.  To compute $a_2$, we apply the skein relation:
$$a_2(K_+) - a_2(K_-) = \lk(L_0)$$
where $K_+$, $K_-$ and $L_0$ are identical except in a neighbourhood of a single crossing, where they differ as shown in Figure~\ref{F:skein}.  Observe that if $K_+$ and $K_-$ are knots (differing by a single crossing change), then $L_0$ is a link of two components.

\begin{figure} [htpb]
    $$\scalebox{1}{\includegraphics{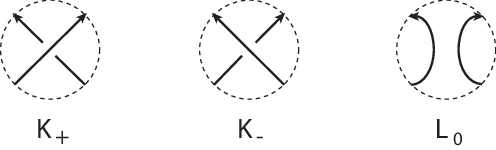}}$$
    \caption{The diagrams $K_+$, $K_-$ and $L_0$.} \label{F:skein}
    \end{figure} 

It is well known that given a diagram of an oriented knot $K$ and a starting point $p$ on the knot, we can change $K$ to a diagram of the unknot by changing crossings so that as we traverse $K$, starting at $p$, the first time we encounter each crossing we cross {\em over} the other strand.  Since $a_2({\rm unknot}) = 0$, this means we can find a sequence of knots $K= K_0, K_1, K_2, \dots, K_m = {\rm unknot}$ such that $K_i$ and $K_{i+1}$ differ by a single crossing change.  So:
$$a_2(K_i) = a_2(K_{i+1}) \pm \lk(L_i) \implies a_2(K) = \sum_{i = 0}^{m-1} \epsilon_i \lk(L_i)$$
where $\epsilon_i = \pm 1$.

This provides a procedure for computing $a_2$ for each cycle in the spatial graph.  For efficiency, we begin by constructing \texttt{boolean[n][n][n][n] overMatrix}, where \texttt{overMatrix[a][b][c][d]} is \texttt{true} if edge $ab$ crosses over edge $cd$, and \texttt{false} otherwise.  Similarly, we construct a \texttt{crossingOrderMatrix} which records for each edge the other edges that cross it, and their order.  Now, for each cycle, we proceed as follows: \begin{enumerate}
	\item Begin traversing the cycle.  At each edge, go through the crossing for that edge (from the \texttt{crossingOrderMatrix}), and determine if any are places where the cycle crosses itself.  If there is a self-crossing, check whether it is an over- or under-crossing, and whether it has been encountered before.  If it is an over-crossing, or has been encountered, move on to the next crossing.
	\item If the crossing is an under-crossing that has not been encountered, change it (temporarily) to an over-crossing.  Then use the order of the crossings along the edges to construct the two links in $L_i$, and compute their linking number as described in Section \ref{SS:link}.  Add the result (with appropriate sign) to the running total for $a_2$.
	\item When all the crossings have been considered, check whether $a_2 = 0$.  If it is non-zero, then the cycle is a knot, and we add it to our list of knotted cycles, and increase the total number of knots by 1.
	\item Undo any changes made to the \texttt{crossingMatrix}, \texttt{overMatrix}, \texttt{crossingOrderMatrix}, etc.  Then move on to the next cycle.
\end{enumerate}

The result is a list of cycles that have non-zero $a_2$, and the number of such cycles.

\small

\normalsize

\end{document}